\documentclass[12pt]{amsart}

\usepackage{amsthm}

\usepackage{amssymb}
\usepackage{verbatim}
\usepackage[curve,matrix,arrow]{xy}
\usepackage{amsmath,amsfonts}
\usepackage{graphicx}
\usepackage{epsf}

\usepackage{amssymb, graphics}

\newcommand{\mathsym}[1]{{}}

\newcommand{\thmref}[1]{Theorem~\ref{#1}}

\newcommand{\lemref}[1]{Lemma~\ref{#1}}

\newcommand{\corref}[1]{Corollary~\ref{#1}}

\newcommand{\figref}[1]{Figure~\ref{#1}}

  {\end{list}}

\def\li{L_{i}}
\def\ri{R_{i}}
\def\ddx{{\frac{d}{dx}}}
\def\oga{{\overline{\ga}}}

\newtheorem{theorem}{Theorem}[section]

\newtheorem{corollary}[theorem]{Corollary}

\newtheorem{lemma}[theorem]{Lemma}

\newtheorem{remark}[theorem]{Remark}

\theoremstyle{definition}

\newcommand{\ga}{\Gamma}

\newcommand{\ee}[1]{E(#1)}
\newcommand{\vv}[1]{V(#1)}
\newcommand{\va}{\upsilon}

\newcommand{\pp}{p_{i}}

\newcommand{\qq}{q_{i}}

\def\<{\langle }
\def\>{\rangle }
\newcommand{\secref}[1]{\S\ref{#1}}

\newcommand{\jm}{\mathrm{J}}

\newcommand{\lm}{\mathrm{L}}
\newcommand{\plm}{\mathrm{L^+}}

\begin{document}

\title[Explicit Rayleigh's Principles and Spanning Trees]
{Explicit Rayleigh's Principles for Resistive Electrical Network and The Total Number of Spanning Trees of Graphs}

\author{Zubeyir Cinkir}
\address{Zubeyir Cinkir\\
Department of Industrial Engineering\\
Abdullah Gul University\\
38100, Kayseri, TURKEY\\}
\email{zubeyir.cinkir@agu.edu.tr}


\thanks{We thank Osman Topan for bringing the reference \cite{Ho} into our attention.}

\begin{abstract}
We give identities for the voltage and resistance functions on a metrized graph to show how these functions behave under any edge deletion/contraction and the identification of any two vertices. This leads to explicit versions of Rayleigh's Principles on a resistive electrical network.
We also establish Euler's Identities for the resistance and the voltage functions on an electrical network. 
One can use these results to study various invariants of metrized graphs and electrical networks. As a specific application, we obtain various identities for the total number of spanning trees of a graph. For example, we show how the total number of spanning trees changes under graph operations such as, contraction of an edge, deletion of an edge, deletion of a vertex, the join of arbitrary two or three vertices.
\end{abstract}

\maketitle

\section{Introduction}\label{sec introduction}

Rayleigh's principles, also known as Short-Cut Principles, in electrical networks are given as follows: 

\textbf{Rayleigh’s Monotonicity Law}\cite[pg. 65]{DS} If the resistances of a circuit
are increased, then the effective resistance between any two points can
only increase. If they are decreased, it can only decrease.

\textbf{Rayleigh’s  Shorting Law}\cite[pg. 99]{DS}  Shorting certain sets of nodes together can only
decrease the effective resistance of the network between two given nodes.

\textbf{Rayleigh’s  Cutting Law}\cite[pg. 99]{DS}  Cutting certain branches can only increase the
effective resistance between two given nodes.

We obtain explicit versions of these laws, which can be found in 
\thmref{thm res cont1},
\thmref{thm del1}, 
\thmref{thm res cont edge} and
\thmref{thm res edge mod}.
We stated these results for a metrized graph but they are also valid for any weighted graph. For example,
\thmref{thm res cont1 on G} below is the corresponding version of \thmref{thm res cont1}.

Let $G$ be a connected weighted graph possibly having self-loops and multiple edges,
and let $G$ be given by the set of vertices $\vv{G}$, containing $n$ vertices, and the set of edges
$\ee{G}$, containing $m$ edges. One can consider $G$ as an electrical network by considering the vertices
as the nodes of the network, and replacing the edges by resistances of magnitude with the corresponding edge length (which is the reciprocal of the weight).

For any two vertices $p, \, q \, \in \vv{G}$,  let $G_{pq}$ be the graph obtained from $G$ by identifying these vertices. For any edge $e \in \ee{G}$, if the deletion of the interior of $e$ does not make $G$ disconnected (i.e., $e$ is not a bridge), we consider the graph $G-e$ obtained from $G$ by deleting the interior of $e$. 

We denote the effective resistance between the nodes $p$ and $q$ of $G$ by $r_G(p,q)$, or simply by $r(p,q)$ if it is clear from the context. Similarly, for any nodes $p$, $q$ and $s$ of $G$ we denote the voltage function $j_p^G(q,s)$, or simply $j_p(q,s)$. When the unit current enters at $s$ and exit from $p$ (with reference voltage $0$ at $p$),  the voltage difference between $q$ and $p$ is given by $j_p(q,s)$.


One of the main results of this paper can be stated for graphs as follows:
\begin{theorem}[Explicit Rayleigh’s  Shorting Law for graphs]\label{thm res cont1 on G}
For any $p, \, q, \, s, \, t \in \vv{G}$, we have
\begin{equation*}\label{}
\begin{split}
r(s,t) &=r_{G_{pq}}(s,t)+\frac{1}{r(p,q)} \big [ j_{p}(q,s) - j_{p}(q,t) \big ]^2.
\end{split}
\end{equation*}
In particular, $\displaystyle{r(s,t) \geq r_{G_{pq}}(s,t) }$ and the equality holds if and only if  $j_{p}(q,s)=  j_{p}(q,t)$.
\end{theorem}
This is a neat result that we wanted to have for more than two decades. One can consult to \cite{AEGM} for a different approach to this problem.

As a consequence of Explicit Rayleigh's Principles, we are able to give refined versions of Euler's identities for the resistance and the voltage functions. These can be found in \thmref{thm res Euler} and \thmref{thm res Euler second}.

Since the resistance values are closely related to $t(G)$ the total number of spanning trees on a graph $G$ (see \thmref{thm res and spantree}), our findings on the resistance function enables us to derive various interesting identities for $t(G)$. These can be found in \secref{sec spanning1} and \secref{sec spanning2}.
In particular, we show how $t(G)$ changes under the join of arbitrary two or three vertices (see \thmref{thm span union2}, \thmref{thm span union3} and \thmref{thm span vertex2and2}).
\corref{cor span vertex contr} explains how contracting an edge effects $t(G)$. Likewise, \thmref{thm span edge del} show how $t(G)$ changes under edge deletion.
We show how $t(G)$ changes if a vertex is deleted (see  \thmref{thm span vertex del0}, \thmref{thm span vertex del1}, \thmref{thm span vertex del2} and \thmref{thm span vertex del general}).


\section{Explicit Rayleigh's Shorting Law}\label{sec metrized graphs}

We first recall minimal information about metrized graphs. If needed, one can consult to \cite{BF}, \cite{BRh}, \cite{C0}, \cite{C2} and \cite{C3}.
It is common to consider the resistance values between vertices of a combinatorial graph. However, the resistance  and the voltage functions are continuous functions on a given resistive electrical network.
We think that metrized graphs are better framework to study these functions. Metrized graphs combines both of the discrete and continuous aspects of these functions. 

There is a $1-1$ correspondence between the equivalence classes of finite connected weighted
graphs, the metrized graphs, and the resistive electric circuits. Therefore, the results stated in this section are valid also for the corresponding weighted graphs.


A \textit{metrized graph} $\ga$ is a finite connected graph
equipped with a distinguished parametri-zation of each of its edges.
In particular, $\ga$ is a one-dimensional manifold except at
finitely many ``branch points''. One can find other definitions of metrized graphs in \cite{RumelyBook}, and in the articles \cite{CR}, \cite{BRh},
and \cite{BF}.

A metrized graph can have multiple edges and self-loops. For any given $p \in \ga$,
the number of directions
emanating from $p$ will be called the \textit{valence} of $p$, and will be denoted by
$\va(p)$. By definition, there can be only finitely many $p \in \ga$ with $\va(p)\not=2$.

For a metrized graph $\ga$, we denote its set of vertices by $\vv{\ga}$.
We require that $\vv{\ga}$ be finite and non-empty and that $p \in \vv{\ga}$ for each $p \in \ga$ if
$\va(p)\not=2$. For a given metrized graph $\ga$, it is possible to enlarge the
vertex set $\vv{\ga}$ by considering more additional points of valence $2$ as vertices.

For a given metrized graph $\ga$, the set of edges of $\ga$ is the set of closed line segments with end points
in  $\vv{\ga}$. We denote the set of edges of $\ga$ by $\ee{\ga}$.
We denote $\# (\vv{\ga})$ and $\# (\ee{\ga})$ by $n$ and $m$, respectively if there is no danger of confusion.
We denote the length of an edge $e_i \in \ee{\ga}$ by $\li$.

A metrized graph $\ga$ is called $r$-regular if it has a vertex set $\vv{\ga}$ such that
$\va(p)=r$ for all vertices $p \in \vv{\ga}$.

For fixed $z$ and $y$ in $\ga$, we consider $\Gamma$ as a resistive
electric circuit with terminals at $z$ and $y$ with the resistance
in each edge given by its length. When unit current enters at $y$
and exits at $z$ (with reference voltage 0 at $z$), $j_{z}(x,y)$ is
the voltage difference between $x$ and $z$, and $r(x,y)=j_{x}(y,y)$
is the resistance between $x$ and $y$.

Both of the resistance function $r(x,y)$ and the voltage function $j_z(x,y)$ are symmetric in $x$ and $y$. Moreover, for each $x, \, y \in \ga$ we have:
\begin{equation}\label{eqn sym}
\begin{split}
&r(x,y) =r(y,x), \quad  j_{z}(x,y)=j_z(y,x)\\
&r(x,x)=0, \quad  j_{x}(y,x)=0\\
&j_{x}(y,y)=r(x,y), \quad r(x,y)=j_x(y,z)+j_y(x,z)\\
& 0 \leq j_{z}(x,y), \quad j_{z}(x,y) \leq r(z,x) \quad \text{and} \quad j_{z}(x,y) \leq r(z,y).
\end{split}
\end{equation}

We recall that the resistance values can be expressed in terms of the entries of the pseudo inverse of the discrete Laplacian matrix.
\begin{lemma} \cite{RB2}, \cite[Theorem A]{D-M} \label{lem disc}
Suppose $G$ is a graph with the discrete Laplacian $\lm$ and the
resistance function $r(x,y)$. For the pseudo inverse $\plm$ of $\lm$, we have
$$r(p,q)=l_{pp}^+-2l_{pq}^+ + l_{qq}^+, \quad \text{for any two vertices $p$ and $q$ of $G$}.$$
\end{lemma}
Similarly, we have
\begin{lemma}\label{lem voltage1}\cite{C1}
Let $\ga$ be a graph with the discrete Laplacian $\lm$ and the voltage
function $j_x(y,z)$. Then for any $p$, $q$, $s$ in $\vv{\ga}$,
$j_p(q,s)=l_{pp}^+ - l_{pq}^+ -l_{ps}^+ +l_{qs}^+.$
\end{lemma}


The following result is one of our main result, usage of which made many of the results in the remaining parts of this paper possible.
\begin{theorem}[Explicit Rayleigh’s  Shorting Law]\label{thm res cont1}
For any $p, \, q, \, s, \, t \in \ga$, we have
\begin{equation*}\label{}
\begin{split}
r(s,t) &=r_{\ga_{pq}}(s,t)+\frac{1}{r(p,q)} \big [ j_{p}(q,s) - j_{p}(q,t) \big ]^2.
\end{split}
\end{equation*}
In particular, $\displaystyle{r(s,t) \geq r_{\ga_{pq}}(s,t) }$ and the equality holds if and only if  $j_{p}(q,s)=  j_{p}(q,t)$.
\end{theorem}
\begin{proof}
We choose a graph model of $\ga$ so that its vertex set contains the points $p, \, q, \, s$ and $t$.
We can apply circuit reductions to $\ga$ so that it becomes equivalent to a graph having only these four vertices  $p, \, q, \, s$ and $t$. Such a graph is illustrated on the left part of
\figref{fig gaandgapq}. This equivalence is realized in the sense that any of the resistance values $r(x,y)$ (and so the voltage values $j_x(y,z)$) among these four vertices agree for two graphs.

Firstly, we consider the discrete Laplacian $L$ of the reduced $\ga$ and its Moore-Penrose inverse $L^+$, and then apply \lemref{lem disc} and \lemref{lem voltage1}. Namely, with the ordering of the vertices
$V=\{t, \, s, \, p, \, q \}$, discrete Laplacian matrix $L$ of the graph (reduced $\ga$) given in the left part of \figref{fig gaandgapq} is as follows:
$$
L=\left[
\begin{array}{cccc}
\frac{1}{a}+\frac{1}{b}+\frac{1}{f}  & -\frac{1}{b} & -\frac{1}{a} & -\frac{1}{f} \\[.15cm]
 -\frac{1}{b} & \frac{1}{b}+\frac{1}{e}+\frac{1}{c} & -\frac{1}{e} & -\frac{1}{c} \\[.15cm]
  -\frac{1}{a} & -\frac{1}{e} & \frac{1}{a}+\frac{1}{e}+\frac{1}{d} & -\frac{1}{d} \\[.15cm]
   -\frac{1}{f} & -\frac{1}{c} & -\frac{1}{d} & \frac{1}{f}+\frac{1}{d}+\frac{1}{e} \\[.15cm]
\end{array}
\right].
$$
Then we can compute the Moore-Penrose inverse $L^+$ of $L$ by using \cite{MMA} with
the following formula (see \cite[ch 10]{C-S}):
\begin{equation}\label{eqn disc5}
 \plm = \big( \lm - \frac{1}{4}\jm \big)^{-1} + \frac{1}{4} \jm.
\end{equation}
where $\jm$ is of size $4 \times 4$ and has all entries $1$. Next, we use $L^+$ and \lemref{lem disc} to obtain $r(s,t)$, $r(p,q)$, $j_{p}(q,s)$, $j_{p}(q,t)$. In this way, we obtain 
\begin{equation}\label{eqn res and vol values1}
\begin{split}
r(s,t) &=\frac{b (a c (d + e) + d e f + a (c + d + e) f + c e (d + f))}{K},\\
r(p,q) &= \frac{a d (c e + b (c + e)) + d ((a + b) c + (a + b + c) e) f}{K},\\
j_{p}(q,s) &= \frac{d e (b f + a (b + c + f))}{K},\\
j_{p}(q,t) &= \frac{a d (b (c + e) + e (c + f))}{K},
\end{split}
\end{equation}
where $K=a b c + a b d + a c d + b c d + a b e + a c e + b d e + c d e + 
 a c f + b c f + a d f + b d f + a e f + b e f + c e f + d e f$.

Secondly, we identify the vertices $p$ and $q$ in $\ga$ (and so in the reduced form of $\ga$) to obtain $\ga_{pq}$. The reduced graph of $\ga_{pq}$ is illustrated in the right part of \figref{fig gaandgapq}.
In this graph, we have
\begin{equation*}\label{eqn res contraction1}
\begin{split}
r_{\ga_{pq}}(s,t)=\frac{1}{\frac{1}{\frac{a f}{a+f}+\frac{c e}{c+e}}+\frac{1}{b}}.
\end{split}
\end{equation*}
Alternatively, we can use $r_{\ga_{pq}}(s,t) = \lim_{d \rightarrow 0} r(s,t)$ and \eqref{eqn res and vol values1} to obtain the same result for $r_{\ga_{pq}}(s,t)$.

Finally, using these values shows that the equality given in the theorem holds.
\end{proof}
\begin{figure}
\centering
\includegraphics[scale=0.53]{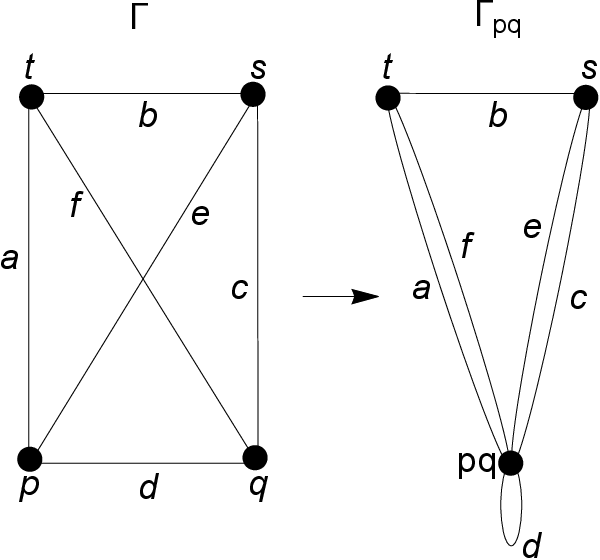} \caption{The graph $\Gamma$ and $\ga_{pq}$ after circuit reductions until four vertices.} \label{fig gaandgapq}
\end{figure}

\begin{theorem}\label{thm res cont2}
For any $p, \, q, \, s \in \ga$, we have
\begin{equation*}\label{}
\begin{split}
r_{\ga_{pq}}(p,s) &=\frac{r(p,s)r(q,s) - j_{s}(p,q)^2}{r(p,q)} 
= j_{s}(p,q) + \frac{j_{p}(q,s)j_{q}(p,s)}{r(p,q)}
=r(p,s)-\frac{j_{p}(q,s)^2}{r(p,q)}.
\end{split}
\end{equation*}
\end{theorem}
\begin{proof}
Setting $t=p$ in \thmref{thm res cont1} and using the fact that $j_p(q,p)=0$ gives
$r_{\ga_{pq}}(p,s)=r(p,s)-\frac{j_{p}(q,s)^2}{r(p,q)}$. The other two equalities follow this one if we use the relations provided in \eqref{eqn sym}. 

Alternatively, we can give a second proof the theorem by circuit reductions as we did in the proof of  \thmref{thm res cont1}.

We choose a graph model of $\ga$ so that its vertex set contains the points $p, \, q$ and $s$.
We can apply circuit reductions to $\ga$ so that it becomes equivalent to a $Y$-shaped graph having these three vertices  $p, \, q$ and $s$ as pendant vertices. Such a graph is illustrated on the left part of \figref{fig gandgpq}.

Next, we identify the vertices $p$ and $q$ in $\ga$ (and so in the reduced form of $\ga$) to obtain 
$\ga_{pq}$. The reduced graph of $\ga_{pq}$ is illustrated in the right part of \figref{fig gandgpq}.
In this graph, we have
\begin{equation*}\label{eqn res contraction2}
\begin{split}
r_{\ga_{pq}}(p,s)=j_{s}(p,q) + \frac{j_{p}(q,s)j_{q}(p,s)}{r(p,q)}.
\end{split}
\end{equation*}
Again, the other equalities in the theorem follows from this one and the relations given in \eqref{eqn sym}.
\begin{figure}
\centering
\includegraphics[scale=0.53]{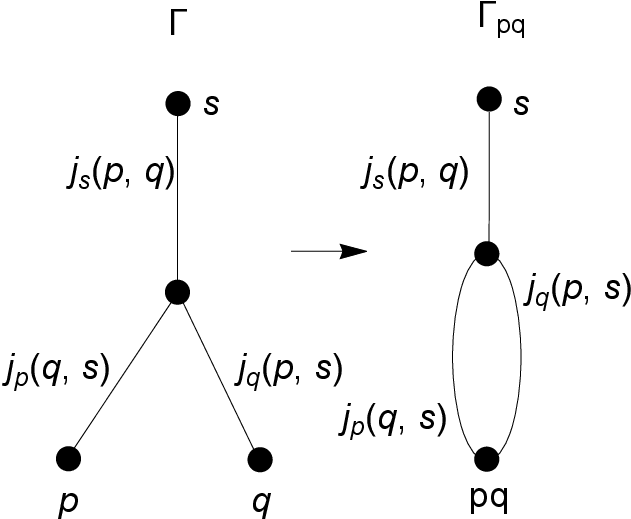} \caption{The graph $\Gamma$ and $\ga_{pq}$ after circuit reductions until three vertices.} \label{fig gandgpq}
\end{figure}
\end{proof}
Note that the relation $r(p,q)=j_p(q,s)+j_q(p,s)$ given in \eqref{eqn sym} can also be derived from the graph shown in the left part of \figref{fig gandgpq}. We can also derive the following relation:
\begin{equation}\label{eqn sym2}
\begin{split}
2j_p(q,s)=r(p,q)+r(p,s)-r(q,s) \quad \text{for each $p$, $q$ and $s$ in $\ga$}
\end{split}
\end{equation}
We can use \eqref{eqn sym} to derive the following equalities (the first one was given in \cite[Thm 8]{BF} as ``Magical Identity"):
\begin{theorem}[Magical Identities]\label{thm magic}
 For any $p$, $q$, $s$ and $t$ in $\ga$, we have
\begin{equation*}\label{eqn sym3}
\begin{split}
j_p(q,s)-j_p(q,t)=j_t(q,s)-j_t(p,s)=j_s(p,t)-j_s(q,t)=j_q(p,t)-j_q(p,s).
\end{split}
\end{equation*}
\end{theorem}
\begin{remark}
We can use \thmref{thm magic} to state the equality in \thmref{thm res cont1} in different forms.
\end{remark}

\begin{theorem}\label{thm vol cont1}
For any $p, \, q, \, s, \, t, \, u \in \ga$, we have
\begin{equation*}\label{}
\begin{split}
j_{u}(t,s) &=j_{u}^{\ga_{pq}}(t,s) +\frac{\big ( j_{p}(q,t)-j_{p}(q,u) \big )^2 
+ \big ( j_{p}(q,s)-j_{p}(q,u) \big )^2 -\big ( j_{p}(q,t)-j_{p}(q,s) \big )^2}{2r(p,q)} \\
&=j_{u}^{\ga_{pq}}(t,s) +\frac{1}{r(p,q)} \big( j_{p}(q,u)-j_{p}(q,t) \big) \big( j_{p}(q,u)-j_{p}(q,s) \big).
\end{split}
\end{equation*}
\end{theorem}
\begin{proof}
The result follows from \thmref{thm res cont1} and \eqref{eqn sym2}.
\end{proof}
Taking $u=p$ in \thmref{thm vol cont1} gives
\begin{equation}\label{eqn sym4}
\begin{split}
j_{p}(t,s) &=j_{p}^{\ga_{pq}}(t,s) +\frac{j_{p}(q,t) j_{p}(q,s)}{r(p,q)}.
\end{split}
\end{equation}
Similarly, setting $t=p$ in \thmref{thm vol cont1} gives
\begin{equation}\label{eqn sym5}
\begin{split}
j_{u}(p,s) &=j_{u}^{\ga_{pq}}(p,s) + \frac{j_{p}(q,u)^2-j_{p}(q,u) j_{p}(q,s)}{r(p,q)}.
\end{split}
\end{equation}

\section{The Deletion and the Contraction Identities for the Resistance Function (Revisited) }\label{sec del and cont}
Let $e_{i}$ be the $i$-th edge, $i \in \{1,2, \dots m\}$, of a given metrized graph $\ga$, and let $\li$ be its length.
We denote its end points by $\pp$ and $\qq$.
Let $\oga_i$ be the graph obtained from $\ga$ by contracting $e_{i}$ to its end points. These points become identical in
$\oga_i$, i.e., $\pp=\qq$ in $\oga_i$.
Likewise, let $\ga-e_i$ be the graph obtained from $\ga$ by deleting the {\em interior of} the edge $e_i$.
We set $\ri:=r_{\ga-e_i}(\pp,\qq)$ when $e_i$ is not a bridge.

In \cite[Thm 4.1]{C4}, we showed the relations between $r(p,q)$ and $r_{\ga-e_i}(p,q)$, and between $r(p,q)$ and $r_{\oga_i}(p,q)$ for any edge $e_i$ and vertices $p$ and $q$ of $\ga$. In this section, we establish improved versions of those relations. These are given in \thmref{thm del1} and \thmref{thm res cont edge} below.

\begin{theorem}[Explicit Rayleigh’s  Cutting Law]\label{thm del1}
Let $e_i$ be an edge with end points $\pp$ and $\qq$. If  $e_i$ is not a bridge of the metrized graph $\ga$,  for any $s, \, t \in \ga-e_i$ we have
\begin{equation*}\label{}
\begin{split}
r(s,t)&=r_{\ga-e_i}(s,t)-\frac{1}{\li + \ri} \big ( j_{\pp}^{\ga-e_i}(\qq,s)-   j_{\pp}^{\ga-e_i}(\qq,t) \big )^2\\
&=r_{\ga-e_i}(s,t)-\frac{\li + \ri}{\li^2} \big ( j_{\pp}(\qq,s)-   j_{\pp}(\qq,t) \big )^2.
\end{split}
\end{equation*}
In particular, $\displaystyle{r_{\ga-e_i}(s,t) \geq r(s,t)}$ and the equality holds if and only if  $j_{\pp}(\qq,s)=  j_{\pp}(\qq,t)$.
\end{theorem}
\begin{proof}
Let $\beta$ be the metrized graph obtained from $\ga$ by adding an edge, of length $L$,  connecting the vertices $\pp$ and $\qq$. Then we have $r_{\beta_{pq}}(s,t)=r_{\ga_{pq}}(s,t)$, $r_{\beta}(p,q)=\frac{L r(p,q)}{L+r(p,q)}$, $j_{p}^{\beta}(q,s)=\frac{L}{L+r(p,q)}j_{p}(q,s)$ and $j_{p}^{\beta}(q,t)=\frac{L}{L+r(p,q)}j_{p}(q,t)$. Moreover, by applying \thmref{thm res cont1} to $\beta$ we have
\begin{equation}\label{eqn beta cont1}
\begin{split}
r_{\beta}(s,t) &=r_{\beta_{pq}}(s,t)+\frac{1}{r_{\beta}(p,q)} \big [ j_{p}^{\beta}(q,s) - j_{p}^{\beta}(q,t) \big ]^2\\
&=r_{\ga_{pq}}(s,t)+\frac{L+r(p,q)}{ L \cdot r(p,q)} \big [ \frac{L}{L+r(p,q)} \big( j_{p}(q,s) - j_{p}(q,t) \big) \big ]^2 \\
&=r(s,t)-\frac{\big( j_{p}(q,s) - j_{p}(q,t) \big)^2}{r(p,q)}+\frac{L}{ r(p,q) (L +r(p,q))} \big( j_{p}(q,s) - j_{p}(q,t) \big)^2\\
&=r(s,t)- \frac{1}{L+r(p,q)}\big( j_{p}(q,s) - j_{p}(q,t) \big)^2,
\end{split}
\end{equation}
where the third equality follows by applying \thmref{thm res cont1} to $\ga$.
Since $\ga$ is obtained by adding the edge $e_i$ to connect $\pp$ and $\qq$ in $\ga-e_i$, the equalities in the theorem follows from \eqref{eqn beta cont1}.
\end{proof}
We can extend \thmref{thm del1} to voltage function:
\begin{theorem}\label{thm vol del1}
Let $e_i$ be an edge with end points $\pp$ and $\qq$. If  $e_i$ is not a bridge of the metrized graph $\ga$,
for any $s, \, t, \, u \in \ga-e_i$ we have,
\begin{equation*}\label{}
\begin{split}
j_{u}(t,s) &=j_{u}^{\ga-e_i}(t,s) -\frac{1}{\li+\ri} \big( j_{\pp}^{\ga-e_i}(\qq,u)-j_{\pp}^{\ga-e_i}(\qq,t) \big) \big( j_{\pp}^{\ga-e_i}(\qq,u)-j_{\pp}^{\ga-e_i}(\qq,s) \big).
\end{split}
\end{equation*}
\end{theorem}
\begin{proof}
The proof follows from \thmref{thm del1} and \eqref{eqn sym2}.
\end{proof}

\begin{theorem}[Explicit Rayleigh's Monotonicity Law I]\label{thm res cont edge}
Let $e_i$ be an edge of $\ga$ with end points $\pp$ and $\qq$.
If $e_i$ is not a bridge of $\ga$, for any $s, \, t \in \ga-e_i$, we have
\begin{equation*}\label{}
\begin{split}
r(s,t) &=r_{\oga_i}(s,t)+\frac{\li+\ri}{\li \cdot \ri} \big [ j_{\pp}(\qq,s) - j_{\pp}(\qq,t) \big ]^2\\
&=r_{\oga_i}(s,t)+\frac{\li}{\ri(\li + \ri)} \big [ j_{\pp}^{\ga-e_i}(\qq,s) - j_{\pp}^{\ga-e_i}(\qq,t) \big ]^2.
\end{split}
\end{equation*}
In particular, $\displaystyle{ r(s,t)  \geq r_{\oga_i}(s,t)}$ and the equality holds if and only if  $j_{\pp}(\qq,s)=  j_{\pp}(\qq,t)$.

If $e_i$ is a bridge of $\ga$, $r(s,t) =r_{\oga_i}(s,t)$ when $s$ and $t$ are on the same connected component of $\ga-e_i$. Otherwise,  $r(s,t) =r_{\oga_i}(s,t)+\li$. 
\end{theorem}
\begin{proof}
Suppose $e_i$ is not a bridge.
We apply \thmref{thm res cont1} with  $s, \, t, \, \pp, \, \qq \in \ga$ to obtain
\begin{equation*}\label{}
\begin{split}
r(s,t) &=r_{\ga_{p_iq_i}}(s,t)+\frac{1}{r(\pp,\qq)} \big [ j_{\pp}(\qq,s) - j_{\pp}(\qq,t) \big ]^2.
\end{split}
\end{equation*}
Then the first equality follows by using the facts that $r_{\ga_{p_iq_i}}(s,t)=r_{\oga_i}(s,t)$ and $r(\pp,\qq)=\frac{\li \ri}{\li+\ri}$. The second equality is obtained from the first one by using the following equalities:
\begin{equation}\label{eqn del1}
\begin{split}
j_{\pp}(\qq,s)=\frac{\li}{\li+\ri}j_{\pp}^{\ga-e_i}(\qq,s), \quad \text{and} \quad 
j_{\pp}(\qq,t)=\frac{\li}{\li+\ri}j_{\pp}^{\ga-e_i}(\qq,t).
\end{split}
\end{equation}
Suppose $e_i$ is a bridge. 

If $s$ and $t$ are on the same connected component of $\ga-e_i$, $r(s,t)$ is independent of $\li$ and so $r(s,t) =r_{\oga_i}(s,t)$. If $s$ and $t$ are on distinct connected components of $\ga-e_i$, $s$ is on the side of $\pp$ or $\qq$. If it is on the side of $\pp$, $r(s,t) =r(s,\pp)+\li+r(\qq,t)=r_{\oga_i}(s,t)+\li$. If it is on the side of $\qq$, $r(s,t) =r(s,\qq)+\li+r(\pp,t)=r_{\oga_i}(s,t)+\li$. 
\end{proof}
We can extend \thmref{thm res cont edge} to voltage functions:
\begin{theorem}\label{thm vol cont2}
Let $e_i$ be an edge with end points $\pp$ and $\qq$. If  $e_i$ is not a bridge of the metrized graph $\ga$,
for any $s, \, t, \, u \in \ga-e_i$ we have,
\begin{equation*}\label{}
\begin{split}
j_{u}(t,s) &=j_{u}^{\oga_i}(t,s) +\frac{\li}{\ri(\li+\ri)} \big( j_{\pp}^{\ga-e_i}(\qq,u)-j_{\pp}^{\ga-e_i}(\qq,t) \big) \big( j_{\pp}^{\ga-e_i}(\qq,u)-j_{\pp}^{\ga-e_i}(\qq,s) \big).
\end{split}
\end{equation*}
\end{theorem}
\begin{proof}
The proof follows from \thmref{thm res cont edge} and \eqref{eqn sym2}.
\end{proof}
Next, we obtain a new proof of \cite[Corollary 4.2]{C4}:
\begin{theorem}\label{thm res del and cont edge}
Let $e_i$ be an edge of $\ga$ with end points $\pp$ and $\qq$.
If  $e_i$ is not a bridge of the metrized graph $\ga$,  for any $s, \, t \in \ga-e_i$ we have
\begin{equation*}\label{}
\begin{split}
r(s,t) &=\frac{\li}{\li + \ri}r_{\ga-e_i}(s,t)+\frac{\ri}{\li + \ri}r_{\oga_i}(s,t).
\end{split}
\end{equation*}
\end{theorem}
\begin{proof}
We multiply both sides of the first equality in \thmref{thm del1} by $\frac{\li}{\li + \ri}$. Then, we multiply both sides of the second equality in \thmref{thm res cont edge} by $\frac{\ri}{\li + \ri}$. Adding these equalities side by side gives the result. 
\end{proof}
The following result is an improved version of \cite[Prop 4.5]{C4}:
\begin{theorem}\label{thm res derivative}
Let $e_i$ be an edge of $\ga$ with end points $\pp$ and $\qq$, and let $\li$ be the length of $e_i$. 
For any $s, \, t \in \ga-e_i$, if $e_i$ is a bridge, 
$\frac{\partial r(s,t)}{\partial L_i}=0$ when $s$ and $t$ are on the same connected component of $\ga-e_i$, and 
$\frac{\partial r(s,t)}{\partial L_i}=1$ when $s$ and $t$ are on distinct connected components of $\ga-e_i$.  

If $e_i$ is not a bridge, for any $s, \, t \in \ga-e_i$ we have
\begin{equation*}\label{}
\begin{split}
\frac{\partial r(s,t)}{\partial L_i}=\frac{1}{(\li+\ri)^2} \big (j_{\pp}^{\ga-e_i}(\qq,s)-j_{\pp}^{\ga-e_i}(\qq,t) \big )^2.
\end{split}
\end{equation*}
\end{theorem}
\begin{proof}
Suppose $e_i$ is not a bridge.

We take derivatives of both sides of the first equality in \thmref{thm del1} or the second equality in  \thmref{thm res cont edge}. Note that $r_{\ga_{p_iq_i}}(s,t)$, $r_{\ga-e_i}(s,t)$, $\ri$, $j_{\pp}^{\ga-e_i}(\qq,s)$ and $j_{\pp}^{\ga-e_i}(\qq,t) $ are independent of $\li$. This completes the proof.

Suppose $e_i$ is a bridge. The result follows from \thmref{thm res cont edge}.
\end{proof}
Next, we give an improved version of \cite[Thm 4.6]{C4}, which is an interesting identity:
\begin{theorem}[Euler's Formula for Resistance Function I]\label{thm res Euler}
Let $e_i$ be an edge of $\ga$ with end points $\pp$ and $\qq$, and let $\li$ be the length of $e_i$. For any $s, \, t \in \vv{\ga}$, we set
$B=\{ e_i \in \ee{\ga} | \text{$s$ and $t$ are disconnected in $\ga-e_i$ } \}$. Then we have
\begin{equation*}\label{}
\begin{split}
r(s,t)&= \sum_{e_i \in B}\li+\sum_{\substack{e_i \in \ee{\ga} \\ \text{$e_i$ is not a bridge}}} \frac{\li}{(\li+\ri)^2} \big (j_{\pp}^{\ga-e_i}(\qq,s)-j_{\pp}^{\ga-e_i}(\qq,t) \big )^2\\
&=\sum_{e_i \in B}\li+\sum_{\substack{e_i \in \ee{\ga} \\ \text{$e_i$ is not a bridge}}} \frac{1}{\li} \big (j_{\pp}(\qq,s)-j_{\pp}(\qq,t) \big )^2.
\end{split}
\end{equation*}
\end{theorem}
\begin{proof}
Since $r(s,t)$ is a continuously differentiable homogeneous function of degree one in the variables $L_1, \, L_2, \, \ldots, 
\, L_m$. Then Euler's formula for such a function is as follows (\cite{C4}):
\begin{equation}\label{eqn res Euler}
\begin{split}
r(s,t)=\sum_{e_i \in \ee{\ga}}\li \frac{\partial r(s,t)}{\partial L_i}.
\end{split}
\end{equation}
Then the result follows from \thmref{thm res derivative} and \eqref{eqn del1}.
\end{proof}
We give a second proof of \thmref{thm res Euler} via harmonic analysis on metrized graphs:
\begin{equation}\label{eqn rpq}
\begin{split}
\int_{\ga} \big( \ddx j_s(t,x) \big )^2 dx&=\int_{\ga}j_s(t,x) \Delta_x j_s(t,x), \quad \text{ by Green's identity \cite{Zh}}\\
&= \int_{\ga}j_s(t,x) \big ( \delta_t(x)-\delta_s(x)  \big ), \quad \text{as $\Delta_x j_s(t,x)=\delta_t(x)-\delta_s(x)$}.\\
&= j_s(t,t) -j_s(t,s)\\
&= r(s,t), \quad \text{by \eqref{eqn sym}}.
\end{split}
\end{equation}
A generalized version of \eqref{eqn rpq} can be found in \cite[Thm 3.19]{C0}. 

On the other hand, $\int_{\ga} \big( \ddx j_s(t,x) \big )^2 dx = \sum_{e_i \in \ee{\ga}} \int_{0}^{\li} \big( \ddx j_s(t,x) \big )^2 dx$.

\begin{figure}
\centering
\includegraphics[scale=0.53]{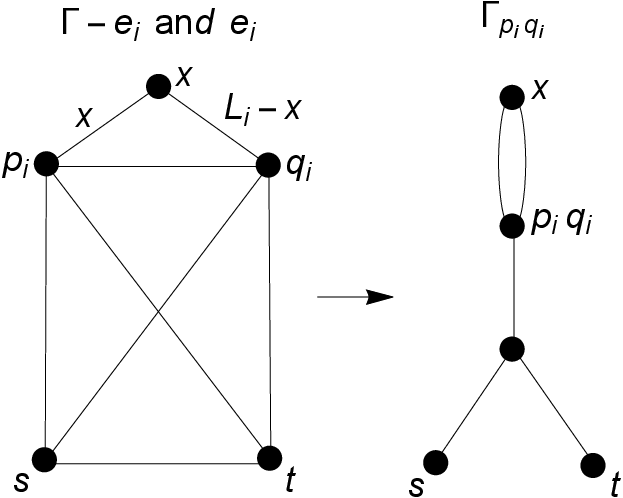} \caption{The graph $\Gamma-e_i$ after circuit reductions until four vertices and $e_i$, and obtaining $\oga_i$.} \label{fig ga3}
\end{figure}

If $e_i$ is not a bridge and $x \in e_i$, we use \thmref{thm vol cont2} to derive
\begin{equation*}\label{}
\begin{split}
j_{s}(t,x) &=j_{s}^{\oga_i}(t,x) +\frac{1}{r(\pp,\qq)} \big( j_{\pp}(\qq,s)-j_{\pp}(\qq,t) \big) \big( j_{\pp}(\qq,s)-j_{\pp}(\qq,x) \big).
\end{split}
\end{equation*}
Thus,
\begin{equation*}\label{}
\begin{split}
\ddx j_{s}(t,x) &=\ddx j_{s}^{\oga_i}(t,x) -\frac{1}{r(\pp,\qq)} \big( j_{\pp}(\qq,s)-j_{\pp}(\qq,t) \big) \ddx j_{\pp}(\qq,x).
\end{split}
\end{equation*}
We have $\ddx j_{s}^{\oga_i}(t,x) =0$, which can be realized if we first consider the circuit reduction of $\ga-e_i$ while keeping $s$, $t$, $\pp$, $\qq$ and $e_i$, and then identify $\pp$ and $\qq$. This is illustrated in \figref{fig ga3}. 

We also have $ j_{\pp}(\qq,x) = \frac{x \ri}{\li +\ri}$ and so $ \ddx j_{\pp}(\qq,x) = \frac{ \ri}{\li +\ri}$. 
This can be obtained if we first consider the circuit reduction of $\ga-e_i$ while keeping $\pp$, $\qq$ and $e_i$, and then applying Delta-Y transformation. This is illustrated in \figref{fig ga4}. 
\begin{figure}
\centering
\includegraphics[scale=0.53]{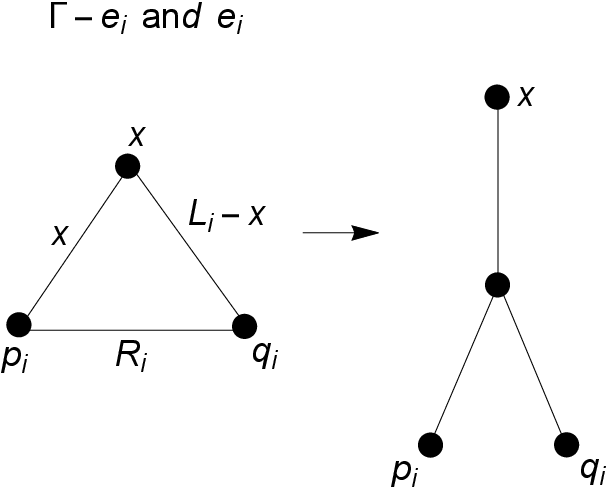} \caption{The graph $\Gamma-e_i$ after circuit reductions until two vertices and $e_i$, and applying Delta-Y transformation.} \label{fig ga4}
\end{figure}

Thus, we have 
\begin{equation*}\label{}
\begin{split}
\ddx j_{s}(t,x) &= \frac{-1}{\li} \big( j_{\pp}(\qq,s)-j_{\pp}(\qq,t) \big).
\end{split}
\end{equation*}
This gives $\int_{0}^{\li} \big( \ddx j_s(t,x) \big )^2 dx = \frac{1}{\li} \big( j_{\pp}(\qq,s)-j_{\pp}(\qq,t) \big)^2$.

If $e_i$ is a bridge, we have two cases. When $s$ and $t$ are on the same connected components of $\ga-e_i$, $j_s(t,x)=j_s(t,\pp)$ or $j_s(t,x)=j_s(t,\qq)$, and so $\ddx j_s(t,x) =0$. When $s$ and $t$ are on distinct connected components of $\ga-e_i$, 
$j_s(t,x)=r(\pp,s)+x$ and so $\ddx j_s(t,x) =1$. Therefore, $\int_{0}^{\li} \big( \ddx j_s(t,x) \big )^2 dx$ is  $\li$ or $0$.

Hence, we have 
\begin{equation}\label{eqn rpq2}
\begin{split}
\int_{\ga} \big( \ddx j_s(t,x) \big )^2 dx= 
\sum_{e_i \in B}\li+\sum_{\substack{e_i \in \ee{\ga} \\ \text{$e_i$ is not a bridge}}} \frac{1}{\li} \big (j_{\pp}(\qq,s)-j_{\pp}(\qq,t) \big )^2.
\end{split}
\end{equation}
Thus the second proof of \thmref{thm res Euler} follows from \eqref{eqn rpq} and \eqref{eqn rpq2}. 
\begin{remark}
We can use \thmref{thm magic} to state the equality in \thmref{thm res Euler} in different forms.
\end{remark}
We can alternatively state \thmref{thm res Euler} in the following form:
\begin{theorem}[Euler's Formula for Resistance Function II]\label{thm res Euler second}
Let $e_i$ be an edge of $\ga$ with end points $\pp$ and $\qq$, and let $\li$ be the length of $e_i$. For any $s, \, t \in \vv{\ga}$, we have
\begin{equation*}\label{}
\begin{split}
r(s,t)=\frac{1}{4}\sum_{e_i \in \ee{\ga}} \frac{1}{\li} \Big[ r(\pp,s)-r(\qq,s)-r(\pp,t)+r(\qq,t) \Big]^2.
\end{split}
\end{equation*}
\end{theorem}
\begin{proof}
We have $j_{\pp}(\qq,s)-j_{\pp}(\qq,t) = \frac{1}{2}\big( r(\pp,s)-r(\qq,s)-r(\pp,t)+r(\qq,t) \big)$ by \eqref{eqn sym2}.
On the other hand, if $e_i$ is a bridge of length $\li$ such that $s$ and $t$ are disconnected in $\ga-e_i$, then
$\big( j_{\pp}(\qq,s)-j_{\pp}(\qq,t) \big)^2 = \frac{1}{4}\big( r(\pp,s)-r(\qq,s)-r(\pp,t)+r(\qq,t) \big)^2=\li^2$.
Then the result follows from \thmref{thm res Euler}.
\end{proof}
Next, we show how resistance value change  explicitly by modifying one of the edge length in $\ga$.
\begin{theorem}[Explicit Rayleigh's Monotonicity Law II]\label{thm res edge mod}
Let $\ga'$ be the metrized graph obtained from $\ga$ by replacing an edge $e_i \in \ee{\ga}$ by $e_i'$ of length $\li'$, and let $\li$ be the length of $e_i$ with end points $\pp$ and $\qq$. Then for any $s$ and $t$ in $\ga-e_i$, if $e_i$ is not a bridge, we have
\begin{equation*}\label{}
\begin{split}
r(s,t)=r_{\ga'}(s,t)+\frac{\li-\li'}{(\li+\ri)(\li'+\ri)} \big (j_{\pp}^{\ga-e_i}(\qq,s) - j_{\pp}^{\ga-e_i}(\qq,t) \big )^2.
\end{split}
\end{equation*}
If $e_i$ is a bridge, $r(s,t)=r_{\ga'}(s,t)$ when $s$ and $t$ are on the same connected component of $\ga-e_i$. Otherwise, 
$r(s,t)=r_{\ga'}(s,t)+\li-\li'$.

In particular, if $e_i$ is not a bridge and $\li> \li'$, then $r(s,t)  \geq r_{\ga'}(s,t)$, and $r(s,t) = r_{\ga'}(s,t)$ if and only if  $j_{\pp}^{\ga-e_i}(\qq,s) = j_{\pp}^{\ga-e_i}(\qq,t)$.
If $e_i$ is a bridge and $\li> \li'$, then $r(s,t)  \geq r_{\ga'}(s,t)$, and $r(s,t) = r_{\ga'}(s,t)$ if and only if $s$ and $t$ are on the same connected component of $\ga-e_i$.
\end{theorem}
\begin{proof}
By applying \thmref{thm res cont edge} to both of $\ga$ and $\ga'$, we obtain the following equalities as $\ga-e_i=\ga'-e_i'$ and $\oga_i=\overline{\ga'_i}$:
\begin{equation*}\label{}
\begin{split}
r(s,t) &=r_{\oga_i}(s,t)+\frac{\li}{\ri(\li + \ri)} \big [ j_{\pp}^{\ga-e_i}(\qq,s) - j_{\pp}^{\ga-e_i}(\qq,t) \big ]^2.\\
r_{\ga'}(s,t)&=r_{\oga_i}(s,t)+\frac{\li'}{\ri(\li' + \ri)} \big [ j_{\pp}^{\ga-e_i}(\qq,s) - j_{\pp}^{\ga-e_i}(\qq,t) \big ]^2.
\end{split}
\end{equation*}
Then the result follows by subtracting the second equation from the first one.

If $e_i$ is a bridge of $\ga$, $r(s,t) =r_{\oga_i}(s,t)=r_{\ga'}(s,t)$ when $s$ and $t$ are on the same connected component of $\ga-e_i$. Otherwise,  $r(s,t) =r_{\oga_i}(s,t)+\li$ and $r_{\ga'}(s,t) =r_{\oga_i}(s,t)+\li'$. 

This completes the proof.
\end{proof}

%
%
%
\section{Spanning Trees and Two Vertex Identification}\label{sec spanning1}

In this section, we consider a connected graph $G$ that have possibly multiple edges or loops. We assume that the edge length $\li$ is $1$ for each edge $e_i$ in the graph.  $G$ contains $n$ vertices and $m$ edges so that $g=m-n+1$ is the genus (i.e., the cyclomatic number) of the graph.

Let $t(G)$ denote the number of spanning trees of a graph $G$.

We recall the following deletion-contraction identity of $t(G)$:
\begin{theorem}\label{thm spantree cont-del}\cite[Prop 4.9]{BM}
For any non-bridge edge $e_i \in \ee{G}$, we have
\begin{eqnarray*}
t(G)=t(\overline{G}_i)+t(G-e_i).
\end{eqnarray*}
\end{theorem}
We recall the following basic well-known cases:
\begin{theorem}\label{thm spantree examples}
Let $P_s$ and $C_s$ denote the path and cycle graphs on $s \geq 2$ vertices, respectively; and let $B_s$ denote the banana graph (or dipole graph) with $s \geq 1$ edges (see \figref{fig pcb}). Then we have
\begin{eqnarray*}
t(P_s)=1, \quad t(C_s)=s \quad \text{and} \quad t(B_s)=s.
\end{eqnarray*}
\end{theorem}
\begin{figure}
\centering
\includegraphics[scale=0.55]{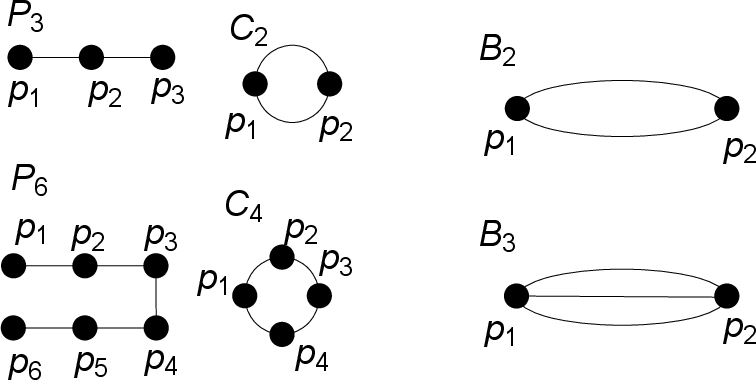} \caption{The graphs $P_3$, $P_6$, $C_2$, $G_4$, $B_2$ and $B_3$.}\label{fig pcb}
\end{figure}

It is well known that the resistance values can be computed via the number of spanning trees:
\begin{theorem}\label{thm res and spantree}
Let $G$ be a connected graph with at least $2$ vertices, and let $p, \, q \in \vv{G}$. Then
\begin{eqnarray*}
r(p,q)=\frac{t(G_{pq})}{t(G)}.
\end{eqnarray*}
\end{theorem}
\thmref{thm res and spantree} was proved by C. Thomassen in 1990 \cite[Proposition 2.3]{T} for any adjacent vertices $p$ and $q$ in a connected graph $G$. In many papers \cite{Kir} and \cite[Section 17]{Big} are cited for the general case, but we think this is somewhat indirectly given. For the convenience of readers, we show that the proof of the general case can be obtained if \thmref{thm res and spantree} is assumed for adjacent vertices.

For any given two vertices  $p$ and $q$ in a connected graph $G$, we consider the graph $\tilde{G}$ that is obtained from $G$ by adding an edge $e$ between $p$ and $q$. Since $p$ and $q$ are adjacent in $\tilde{G}$, assuming \thmref{thm res and spantree} for adjacent vertices gives
\begin{eqnarray*}
r_{\tilde{G}}(p,q)=\frac{t(\tilde{G}_{pq})}{t(\tilde{G})}.
\end{eqnarray*}
On the other hand, we have $r_{\tilde{G}}(p,q)=\frac{r(p,q)}{1+r(p,q)}$, $t(\tilde{G}_{pq})=t(G_{pq})$, $t(\tilde{G})=t(\tilde{G}_{pq})+t(\tilde{G}-e)$ and $\tilde{G}-e=G$. Thus, we have
\begin{eqnarray*}
\frac{r(p,q)}{1+r(p,q)}=\frac{t(G_{pq})}{t(G_{pq})+t(G)}.
\end{eqnarray*}
Taking the reciprocal of both sides of this equation gives \thmref{thm res and spantree} for any $p$ and $q$.

Note that we set $t(G_{pp}):=0$ to make the equality in \thmref{thm res and spantree} hold when $p=q$ as $r(p,p)=0$. When $G$ has only the vertices $p$ and $q$, $G_{pq}$ has only one vertex $pq$, but then the equality in \thmref{thm res and spantree} holds if and only if $t(G_{pq})=1$.
Therefore, we set  $t(G):=1$ when $G$ has only one vertex. 

We can extend \thmref{thm res and spantree} to voltage values by combining \eqref{eqn sym2} and \thmref{thm res and spantree}:
\begin{theorem}\label{thm vol and spantree}
Let $G$ be a connected graph, and let $p, \, q, \, s \in \vv{G}$. Then
\begin{eqnarray*}
j_p(q,s)=\frac{t(G_{pq})+t(G_{ps})-t(G_{qs})}{2t(G)}.
\end{eqnarray*}
\end{theorem}

If $e_i$ is an edge with end points $\pp$ and $\qq$, we have $r(\pp,\qq)=\frac{\li \ri}{\li+\ri}$ with $\ri=r_{G-e_i}(\pp,\qq)$ and $t(G_{\pp \qq})=t(\overline{G}_i)$. Thus, if $\li =1$, then we derive the following equalities by using \thmref{thm spantree cont-del} and \thmref{thm res and spantree}:
\begin{equation}\label{eqn ri and tgi}
\begin{split}
\ri=\frac{t(\overline{G}_i)}{t(G-e_i)}, \quad \quad \frac{\ri}{\li+\ri}=\frac{t(\overline{G}_i)}{t(G)}, \quad \text{and} \quad
\frac{\li}{\li+\ri}=\frac{t(G-e_i)}{t(G)}.
\end{split}
\end{equation}
Next, we give averaging results of $t(G)$ for edge contractions and edge deletions:
\begin{theorem}\label{thm spantree average}
Let $G$ be a connected graph. Then
\begin{eqnarray*}
t(G)=\frac{1}{n-1}\sum_{e_i \in E(G)}t(\overline{G}_i).
\end{eqnarray*}
Morevover, if G has no bridges, we have
\begin{eqnarray*}
t(G)=\frac{1}{g}\sum_{e_i \in E(G)}t(G-e_i).
\end{eqnarray*}
\end{theorem}
\begin{proof}
We provide two proofs. 

Firstly, we have $\sum_{e_i \in \ee{G}}\frac{\ri}{\li+\ri}=n-1$ \cite[pg. 26]{CR}, which is equivalent to the Foster's first identity (see \cite{C1} and references therein). Then 
\begin{equation*}\label{}
\begin{split}
n-1 &=\sum_{e_i \in \ee{G}}\frac{\ri}{\li+\ri} =\sum_{e_i \in \ee{G}}r(\pp,\qq), \quad \text{as $\li =1$ for each edge $e_i$}\\ 
&= \sum_{e_i \in \ee{G}} \frac{t(\overline{G}_i)}{t(G)}, \quad \text{by \thmref{thm res and spantree}}\\ 
&= \frac{1}{t(G)} \sum_{e_i \in \ee{G}} t(\overline{G}_i).
\end{split}
\end{equation*}
This gives the proof of the first equality in the theorem. On the other hand, by summing up the identity in \thmref{thm spantree cont-del} over all edges gives
\begin{equation*}\label{}
\begin{split}
\sum_{e_i \in \ee{G}}t(G) =\sum_{e_i \in E(G)}t(\overline{G}_i)+\sum_{e_i \in \ee{G}}t(G-e_i)
\end{split}
\end{equation*}
which reduces to
\begin{equation*}\label{}
\begin{split}
m \cdot t(G) =(n-1) t(G)+\sum_{e_i \in \ee{G}}t(G-e_i).
\end{split}
\end{equation*}
This proves the second identity in the theorem as $g=m-n+1$.

The second proof is obtained by applying \thmref{thm res and spantree} and \eqref{eqn ri and tgi} into \cite[Theorem 4.7]{C4}. For any non-adjacent vertices $s$ and $t$, \cite[Theorem 4.7]{C4} gives
\begin{equation*}
\begin{split}
r(s,t)=\frac{1}{n-2} \sum_{e_i \in \ee{G}} \frac{\ri}{\li+\ri} r_{\overline{G}_i}(s,t).
\end{split}
\end{equation*}
Using \thmref{thm res and spantree} and \eqref{eqn ri and tgi} gives
\begin{equation*}
\begin{split}
\frac{t(G_{st})}{t(G)}=\frac{1}{n-2} \sum_{e_i \in \ee{G}} \frac{t(\overline{G}_i)}{t(G)} \frac{t((\overline{G}_i)_{st})}{t(\overline{G}_i)}.\\
\end{split}
\end{equation*}
We have $(\overline{G}_i)_{st}=\overline{(G_{st})_i}$. Since  $s$ and $t$ are not adjacent, $\ee{G}=\ee{G_{st}}$. Thus we can rewrite this equality as follows:
\begin{equation*}
\begin{split}
t(G_{st})=\frac{1}{n-2} \sum_{e_i \in \ee{G_{st}}}t(\overline{(G_{st})_i}),\\
\end{split}
\end{equation*}
where $G_{st}$ has $n-1$ vertices. Since $s$, $t$ and $G$ were chosen arbitrary, the last equality is equivalent to the first equality in the theorem.

Let $G$ has no bridges. Again, for any non-adjacent vertices $s$ and $t$, \cite[Theorem 4.7]{C4} gives
\begin{equation*}
\begin{split}
r(s,t)=\frac{1}{g+1} \sum_{e_i \in \ee{G}} \frac{\li}{\li+\ri} r_{G-{e_i}}(s,t).
\end{split}
\end{equation*}
Using \thmref{thm res and spantree} and \eqref{eqn ri and tgi} gives
\begin{equation*}
\begin{split}
\frac{t(G_{st})}{t(G)}=\frac{1}{g+1} \sum_{e_i \in \ee{G}} \frac{t(G-e_i)}{t(G)} \frac{t((G-e_i)_{st})}{t(G-e_i)}.\\
\end{split}
\end{equation*}
We have $(G-e_i)_{st}=G_{st}-e_i$. Since  $s$ and $t$ are not adjacent, $\ee{G}=\ee{G_{st}}$. Thus we can rewrite this equality as follows:
\begin{equation*}
\begin{split}
t(G_{st})=\frac{1}{g+1} \sum_{e_i \in \ee{G_{st}}}t(G_{st}-e_i),\\
\end{split}
\end{equation*}
where $G_{st}$ has genus $g+1$ as it has $n-1$ vertices and $m$ edges. Since $s$, $t$ and $G$ were chosen arbitrary, the last equality is equivalent to the second equality in the theorem.

This completes the second proof of the theorem.
\end{proof}

\textbf{Example:} Let $G=K_n$ be a complete graph on $n$ vertices. Then it has $m=\frac{n(n-1)}{2}$ edges, and so $g=\frac{(n-2)(n-1)}{2}$. By the symmetry in $K_n$, we have $t(G-e_i)=t(G-e_j)$ and $t(\overline{G}_i)=t(\overline{G}_j)$ for any $e_i$ and $e_j$. Since $t(K_n)=n^{n-2}$ by Cayley's theorem \cite{Cay}, \thmref{thm spantree average} gives
$t(\overline{G}_i)=2 n^{n-3}$ and $t(G-e_i)=n^{n-3}(n-2)$ for any edge $e_i$ of $K_n$.

Next, we recall the following well-known fact:
\begin{theorem}\label{thm span union1}
Let $G=G_1 \cup G_2$ with $G_1 \cap G_2 =\{ p \}$ for a vertex $p$ of $G$. Then
\begin{equation*}
\begin{split}
t(G)=t(G_1) \cdot t(G_2)
\end{split}
\end{equation*}
Moreover, if $\displaystyle{G=\bigcup_{i=1}^{k} G_{i}}$ and  $G_i \cap G_j =\{ p \}$ for every distinct $i$ and $j$ in $\{1, \, 2, \, \ldots, \, k  \}$, then we have
\begin{equation*}
\begin{split}
t(G)=\prod_{i=1}^k t(G_i).
\end{split}
\end{equation*}
\end{theorem}
\begin{proof}
For every spanning tree of $G_1$ we have $t(G_2)$ choices in $G_2$ to form a spanning tree of $G$. Thus, the first result follows by the multiplication property of counting. The second equality follows from the successive application of the first one.
\end{proof}
We can generalize \thmref{thm span union1} as follows:
\begin{theorem}\label{thm span union1g}
Let $T$ be a tree graph with $n+1$ vertices, and let $G$ be the graph obtained from $T$ by replacing each of its edge $e_i$ with end points $\pp$ and $\qq$ by a graph $G_i$ in such a way that $\pp$ and $\qq$ are identified by the two vertices $u_i$ and $v_i$ of $G_i$. Then
\begin{equation*}
\begin{split}
t(G)=\prod_{i=1}^{n} t(G_i).
\end{split}
\end{equation*}
\end{theorem}
\begin{proof}
If $e_n$ is the edge of $T$ connected to a vertex of degree $1$, then $t(G)=t(G_n) \cdot t(H)$ by \thmref{thm span union1}, where $H$ is the graph obtained from $T-e_n$ in the same way $G$ is obtained from $T$. Thus the result follows by applying this idea successively.
\end{proof}

The formulas in \thmref{thm span union1} and \thmref{thm span union1g} are applicable if $G$ has a vertex whose removal (removing the vertex and the edges connected to it) disconnects $G$. In that case, $G$ has vertex connectivity $1$. Next, we consider the graphs with vertex connectivity $2$ and obtain a formula for the number of spanning trees.
\begin{theorem}\label{thm span union2}
Let $G=G_1 \cup G_2$ with $G_1 \cap G_2 =\{ p, \, q \}$ for any vertices $p$ and $q$ of $G$. Then
\begin{equation*}
\begin{split}
t(G)=t(G_1) \cdot t(G_{2,pq})+t(G_2) \cdot t(G_{1,pq}), 
\end{split}
\end{equation*}
where $G_{j, pq}$ is obtained from $G_j$ by identifying the vertices $p$ and $q$ in $G_j$ for any $j \in \{ 1, \, 2 \}$.
\end{theorem}
\begin{proof}
We have $r(p,q)=\frac{r_{G_1}(p,q) r_{G_2}(p,q)}{r_{G_1}(p,q)+r_{G_2}(p,q)}$ by parallel circuit reduction.
We apply \thmref{thm res and spantree} to obtain
\begin{eqnarray*}
\frac{t(G_{pq})}{t(G)}=\frac{\frac{t(G_{1, \, pq})}{t(G_1)} \frac{t(G_{2, \, pq})}{t(G_2)}}{\frac{t(G_{1, \, pq})}{t(G_1)} + \frac{t(G_{2, \, pq})}{t(G_2)}}.
\end{eqnarray*}
On the other hand, $t(G_{pq})=t(G_{1, \, pq}) t(G_{2, \, pq})$ by \thmref{thm span union1} as the identified vertex $pq$ is a cut vertex in $G_{pq}$ (i.e., removal of $pq$ disconnects $G_{pq}$).
Then the result follows from these equalities.
\end{proof}
\begin{corollary}\label{cor span union2}
Let $G=G_1 \cup P_{k+1}$ with $G_1 \cap P_{k+1} =\{ p, \, q \}$ for any vertices $p$ and $q$ of a graph $G_1$, and let
$p$ and $q$ be the vertices having degree $1$ in the path graph $P_{k+1}$. If $k \geq 2$, then
\begin{equation*}
\begin{split}
t(G)=k \cdot t(G_{1})+ t(G_{1,pq}), 
\end{split}
\end{equation*}
\end{corollary}
\begin{proof}
The result follows from \thmref{thm span union2} and \thmref{thm spantree examples}.
\end{proof}

\textbf{Example:} Let $G$, $G_1$, $G_2$, $p$ and $q$ be as in the first group of graphs in \figref{fig g5andg6s}. If we consider $G$ as the union of two graphs along the cut vertex $s$ in $G$, we have $t(G)=6 \cdot 2$ by \thmref{thm span union1} and \thmref{thm spantree examples}. On the other hand, if we consider $G=G_1 \cup G_2$ with $G_1 \cap G_2 = \{p, \, q \}$. Then $t(G)=1 \cdot 6 + 2 \cdot 3=12$ by \thmref{thm span union2} as $t(G_1)=1$, $t(G_{2, \, pq})=3 \cdot 2 =6$, $t(G_2)=2$, $t(G_{1, \, pq})=3$.

\textbf{Example:} Let $G$, $G_1$, $G_2$, $p$ and $q$ be as in the second group of graphs in \figref{fig g5andg6s}. Then $t(G)=8 \cdot 21 + 13 \cdot 8=272$ by \thmref{thm span union2} as $t(G_1)=8$, $t(G_{2, \, pq})=21$, $t(G_2)=13$, $t(G_{1, \, pq})=8$. Here, we computed each of $t(G_1)$, $t(G_{1, \, pq})$, $t(G_2)$, $t(G_{2, \, pq})$ by using \thmref{thm span union2}, \thmref{thm span union1}, \thmref{thm spantree cont-del} and \thmref{thm spantree examples}.
%
\begin{figure}
\centering
\includegraphics[scale=0.55]{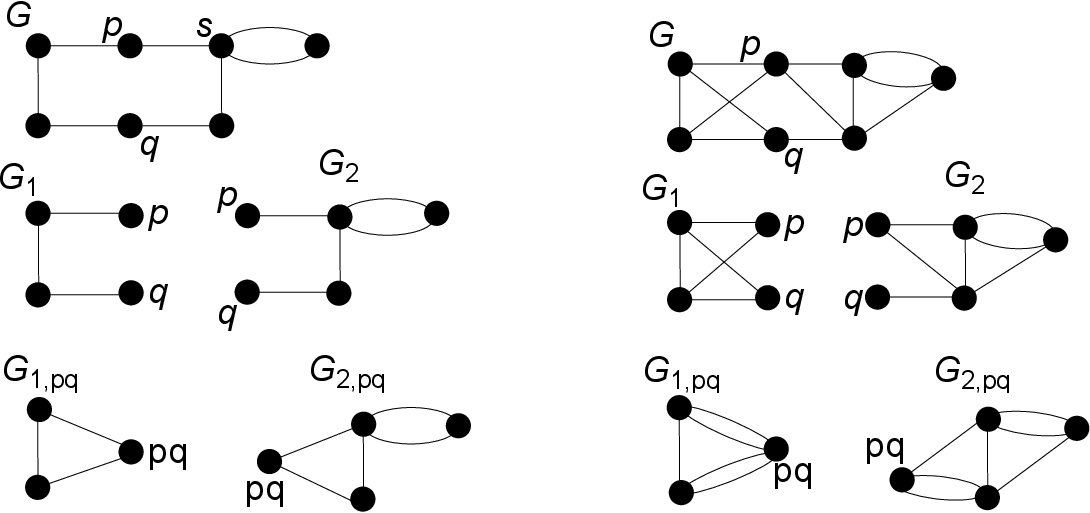} \caption{The graphs $G$, $G_1$, $G_2$, $G_{1, \, pq}$ and $G_{2, \, pq}$.}\label{fig g5andg6s}
\end{figure}

\begin{theorem}\label{thm span vertex del0}
Let $p$, $q$, and $u$ be vertices of a graph $G$, and let $u$ be adjacent to the vertices $p$ and $q$ via by $a$ and $b$ number of edges, respectively.
Then, we have
\begin{equation*}
\begin{split}
t(G)=(a+b) t(H)+ a b \, t(H_{pq}), 
\end{split}
\end{equation*}
where $H$ is the graph obtained from $G$ by deleting the vertex $u$.
\end{theorem}
\begin{proof}
Note that the deletion of a vertex requires also the deletions of all edges that are adjacent to that vertex.
We take $G_1=H$ and $G_2$ as the graph with only the vertices $p$, $q$ and $u$ such that there are $a$ edges between $p$ and $u$, $b$ edges between $q$ and $u$. Then $G=G_1 \cup G_2$ with 
$G_1 \cap G_2 =\{ p, \, q \}$.
Since $t(G_2)=a b$, $t(G_{2,pq})=a+b$, the result follows from \thmref{thm span union2}.
\end{proof}
We can generalize \thmref{thm span union2} in several ways. For example:
\begin{theorem}\label{thm span union3}
Let $\displaystyle{G=\bigcup_{i=1}^{k} G_{i}}$ with $G_i \cap G_j =\{ p, \, q \}$ for every distinct $i$ and $j$ in $\{1, \, 2, \, \ldots, \, k  \}$. Then we have
\begin{equation*}
\begin{split}
t(G)= \prod_{j=1}^k t(G_{j, \, pq}) \sum_{i=1}^{k} \frac{t(G_i)}{t(G_{i, \, pq})},
\end{split}
\end{equation*}
where $G_{j, pq}$ is obtained from $G_j$ by identifying the vertices $p$ and $q$ in $G_j$ for any $j \in \{ 1, \, 2, \, \ldots, \, k \}$.

In particular, if each $G_i$ is a copy of the same graph $G_1$, then
$t(G)=k \cdot t(G_1) \cdot t(G_{1, \, pq})^{k-1}$.
\end{theorem}
\begin{proof}
We have  $\displaystyle{\frac{1}{r(p,q)}= \sum_{i=1}^{k} \frac{1}{r_{G_i}(p,q)}}$
by parallel circuit reduction.
We apply \thmref{thm res and spantree} to obtain
\begin{eqnarray*}
\frac{t(G)}{t(G_{pq})}= \sum_{i=1}^{k} \frac{t(G_i)}{t(G_{i, \, pq})}.
\end{eqnarray*}
On the other hand, $\displaystyle{t(G_{pq})=\prod_{j=1}^k t(G_{j, \, pq}) }$ by \thmref{thm span union1} as the identified vertex $pq$ is a cut vertex in $G_{pq}$.
Then the result follows from these equalities.
\end{proof}
Note that the graph $G$ in \thmref{thm span union3} can be seen as the graph $BG_k$ which is obtained from the banana graph $B_k$ by replacing each of its edges by a graph $G_i$ in such a way that the given two vertices of $G_i$ are identified by the two vertex of $B_k$. 

Another generalization of \thmref{thm span union2} can be given as follows:
\begin{theorem}\label{thm span union4}
Let $CG_n$ be the graph obtained from the cycle graph $C_n$ with vertex set $\{ p_1, \, \ldots \, , p_n \}$ by replacing each edge $e_i$  by a graph $G_i$ in such a way that the vertices of $e_i$ are identified with the given two vertices $s_i$ and $t_i$ of $G_i$. Then we have
\begin{equation*}
\begin{split}
t(CG_n)= \prod_{i=1}^n t(G_{i}) \sum_{i=1}^{n} \frac{t(G_{i, \, s_i t_i})}{t(G_{i})},
\end{split}
\end{equation*}
where $G_{i, s_i t_i}$ is obtained from $G_i$ by identifying the vertices $s_i$ and $t_i$ in $G_i$ for any $i \in \{ 1, \, 2, \, \ldots, \, n \}$.

In particular, if each $G_i$ is a copy of the same graph $G$ with fixed vertices $s$ and $t$, then
$t(CG_n)=n \cdot t(G)^{n-1} \cdot t(G_{st})$.
\end{theorem}
\begin{proof} We provide two proofs.
The first one is by induction on $n$. 

When $n=2$, the result follows from \thmref{thm span union2}. Because, we can view $CG_2$ as $G_1 \cup G_2$  with $G_1 \cap G_2 =\{ p_1, \, p_2 \}$ so that the vertices $s_1$ and $t_1$ of $G_1$ are identified with $p_1$ and $p_2$, respectively and that the vertices $s_2$ and $t_2$ of $G_2$ are identified with $p_2$ and $p_1$, respectively.

Assume that the given identity holds for some integer $k \geq 2$. We note that $CG_{k+1}=  G_{k+1} \cup H$ with $G_{k+1} \cap H =\{ p_{k+1}, \, p_{1} \}$, where $H_{p_{k+1} p_1}$ is nothing but $CG_{k}$. Moreover, $t(H)=\prod_{i=1}^k t(G_{i})$ by \thmref{thm span union1} as each vertex $p_i$ is a cut vertex in $H$ for all $i \in \{2, \, \ldots, \, k  \}$. Then we have
\begin{equation*}
\begin{split}
t(CG_{k+1})&= t(G_{k+1}) \cdot t(H_{pq})+t(H) \cdot t(G_{k+1,s_{k+1}t_{k+1}}), \quad \text{by \thmref{thm span union2}}\\
&=t(G_{k+1}) \cdot t(CG_{k})+\prod_{i=1}^k t(G_{i}) \cdot t(G_{k+1,s_{k+1}t_{k+1}}), \quad \text{by the observations above}\\
&=t(G_{k+1}) \cdot \prod_{i=1}^k t(G_{i}) \sum_{i=1}^{k} \frac{t(G_{i, \, s_i t_i})}{t(G_{i})}+\prod_{i=1}^k t(G_{i}) \cdot t(G_{k+1,s_{k+1}t_{k+1}})\\
&=\prod_{i=1}^{k+1} t(G_{i}) \sum_{i=1}^{k} \frac{t(G_{i, \, s_i t_i})}{t(G_{i})}+\prod_{i=1}^{k+1} t(G_{i}) \cdot 
\frac{t(G_{k+1,s_{k+1}t_{k+1}})}{t(G_{k+1})},
\end{split}
\end{equation*}
where the third equality follows from the induction assumption. This proves the given identity for $k+1$.
Then the result follows from the principle of induction.

The second proof is based on the fact that $CG_n$ can be seen as $G_{p_1 p_{n+1}}$, where $G$ is the graph obtained from the path graph $P_{n+1}$
by replacing each of its edge $e_i$ by the graph $G_i$ so that the vertices $p_i$ and $p_{i+1}$ of $P_{n+1}$ are identified by the vertices $s_i$ and $t_i$ in $G_i$. In $G$, we have $\displaystyle{r_G(p_1,p_{n+1})=\sum_{i=1}^{n}r_{G_i}(s_i,t_i)}$ by series connection. Then by \thmref{thm res and spantree} we have 
$$
\frac{t(G_{p_1p_{n+1}})}{t(G)}=\sum_{i=1}^{n} \frac{t(G_i,s_it_i)}{t(G_i)}.
$$
On the other hand, $\displaystyle{t(G)=\prod_{i=1}^{n} t(G_{i})}$ by \thmref{thm span union1}. 
Thus, 
$$t(G_{p_1p_{n+1}})=\prod_{i=1}^{n} t(G_{i})\sum_{i=1}^{n} \frac{t(G_i,s_it_i)}{t(G_i)}.$$
Then the result follows by noting that $G_{p_1p_{n+1}}$ is nothing but $CG_n$. 
\end{proof}

We can use  \thmref{thm span union1g},   \thmref{thm span union3} and \thmref{thm span union4} in various combinations to find the number of spanning trees of more complicated graphs. One possible usage is given as follows: 
\begin{corollary}\label{cor replace banana}
Let $G_i$ be the graph obtained from the path graph $P_{n_i+1}$ by replacing each of its edge $e_j$ by a graph $G_{i,j}$ in such a way that 
the end points $p_j$ and $p_{j+1}$ of $e_j$ are identified by the given two vertices $s_j$ and $t_j$ in  $G_{i,j}$ for each $j =1, \, \ldots, \, n_i$. For the graph $BG_k$ obtained from $B_k$ with vertices $p$ and $q$ by replacing each of its edges by the graph $G_i$ so that $p_1 \in G_i$ is identified by $p$ and  $p_{n_i+1} \in G_i$ is identified by $q$ for each  $i =1, \, \ldots, \, k$. Then we have
\begin{equation*}
\begin{split}
t(BG_k)= \prod_{i=1}^k \Big[ \prod_{j=1}^{n_i} t(G_{i,j}) \sum_{j=1}^{n_i} \frac{t(G_{i, j, \, s_j t_j})}{t(G_{i,j})}  \Big]
\sum_{i=1}^{k} \frac{\prod_{j=1}^{n_i}t(G_{i, j})}{\prod_{j=1}^{n_i} t(G_{i,j}) \sum_{j=1}^{n_i} \frac{t(G_{i, j, \, s_j t_j})}{t(G_{i,j})}}
\end{split}
\end{equation*}
In particular, if each $G_{i,j}$ is the same graph $H$ given with fixed vertices $s$ and $t$, and that $n_i=n$, then we have
$t(BG_k)= k \cdot n^{k-1} t(H)^{(n-1)(k-1)+n} t(H_{st})^{k-1}.$ 
\end{corollary}
\begin{proof}
By \thmref{thm span union3},
\begin{equation*}
\begin{split}
t(BG_k)= \prod_{i=1}^k t(G_{i, \, pq}) \sum_{i=1}^{k} \frac{t(G_i)}{t(G_{i, \, pq})},
\end{split}
\end{equation*}
and by \thmref{thm span union4} we have  
\begin{equation*}
\begin{split}
t(G_{i, \, pq})= \prod_{j=1}^{n_i} t(G_{i,j}) \sum_{j=1}^{n_i} \frac{t(G_{i, j, \, s_j t_j})}{t(G_{i,j})},
\end{split}
\end{equation*}
Moreover, by \thmref{thm span union1g}, $\displaystyle{t(G_i)=\prod_{j=1}^{n_i}t(G_{i, j})}$.
Then the result follows from these equalities.
\end{proof}
Here is an example for the graph considered in \corref{cor replace banana} with $k=3$, $n_1=1$, $n_2=2$, $n_3=3$.

\textbf{Example:} Let $BG_3$ be as in the first graph in \figref{fig bg}. Then $G_1$, $G_2$, $G_3$, $G_{2,1}$, $G_{2,2}$, $G_{3,1}$, $G_{3,2}$ and $G_{3,3}$ are illustrated in the same figure. Note that $t(G_1)=8$, $t(G_{1,pq})=8$, $t(G_{2,1})=2$, $t(G_{2,1,s_1t_1})=1$, 
$t(G_{2,2})=3$, $t(G_{2,2,s_2t_2})=2$,
$t(G_{3,1})=1$, $t(G_{3,1,s_1t_1})=4$,
$t(G_{3,2})=5$, $t(G_{3,2,s_2t_2})=6$,
$t(G_{3,3})=3$, $t(G_{3,3,s_3t_3})=1$. Thus, by \corref{cor replace banana} $t(BG_3)=9472$.

\begin{figure}
\centering
\includegraphics[scale=0.55]{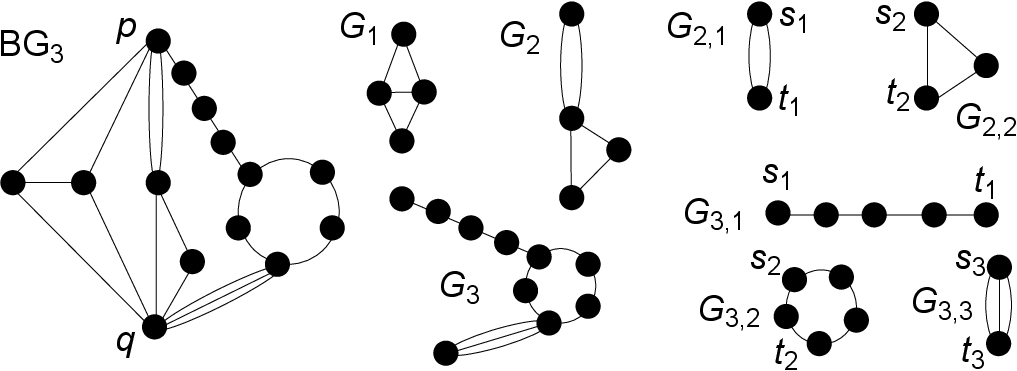} \caption{The graphs $BG_3$ and the related subgraphs.}\label{fig bg}
\end{figure}

The following result is the analogue of \thmref{thm res Euler} for the number of spanning trees. That is why we have the name.
\begin{theorem}[Euler's formula for the number of spanning trees]\label{thm span Euler}
Let $e_i$ be an edge of $G$ with end points $\pp$ and $\qq$. For any $s, \, t \in \vv{G}$, we have
\begin{equation*}\label{}
\begin{split}
t(G_{st})&= k \cdot t(G) + \frac{1}{4 t(G)}\sum_{\substack{e_i \in \ee{G} \\ \text{$e_i$ is not a bridge}}}\big [ t(G_{\pp s})-t(G_{\pp t}) -t(G_{\qq s})+t(G_{\qq t}) \big ]^2,
\end{split}
\end{equation*}
where $k$ is the number of bridges on any path connecting $s$ and $t$. Equivalently, we have
\begin{equation*}\label{}
\begin{split}
t(G_{st})&= \frac{1}{4 t(G)}\sum_{e_i \in \ee{G}}\big [ t(G_{\pp s})-t(G_{\pp t}) -t(G_{\qq s})+t(G_{\qq t}) \big ]^2,
\end{split}
\end{equation*}
\end{theorem}
\begin{proof}
The first equality follows from \thmref{thm res Euler}, \thmref{thm res and spantree} and \thmref{thm vol and spantree}.
The second equality follows from \thmref{thm res Euler second} and \thmref{thm res and spantree}.
\end{proof}

\section{Spanning Trees and More Than Two Vertex Identification}\label{sec spanning2}

$G_{pq,st}$ means $s$ and $t$ are identified with each other in $G_{pq}$. 

$G_{pqs}$ means the vertices $p$, $q$ and $s$ are identified with each other in $G$.

In general, $G_{p_1p_2 \ldots p_i,q_1 q_2 \ldots q_j, \, \ldots, \, s_1 s_2 \ldots s_k}$ is used to denote a graph obtained from the graph $G$ by identifiying the set of vertices $\{ p_1, \, p_2, \, \ldots, \, p_i\}$ into a vertex $\bar{p}$, the set of vertices $\{ q_1, \, q_2, \, \ldots, \, q_j\}$ into a vertex $\bar{q}$ and so on.

The following identities are the main results of this section: 
\begin{theorem}\label{thm span vertex2and2}
For any $s, \, t, \, p, \, q \in \vv{G}$, we have
\begin{equation*}\label{}
\begin{split}
t(G) t(G_{pq,st})=t(G_{st})t(G_{pq}) - \frac{1}{4}\big[ t(G_{p s})-t(G_{q s}) -t(G_{p t})+t(G_{q t}) \big ]^2.
\end{split}
\end{equation*}
In particular, for any $s, \, p, \, q \in \vv{G}$, we have
\begin{equation*}\label{}
\begin{split}
t(G) t(G_{pqs})=t(G_{ps})t(G_{pq}) - \frac{1}{4}\big[ t(G_{p s})-t(G_{q s}) +t(G_{pq}) \big ]^2.
\end{split}
\end{equation*}
\end{theorem}
\begin{proof}
The first equality follows from \thmref{thm res cont1}, \thmref{thm res and spantree} and \thmref{thm vol and spantree}.
If we set $t=p$ in the first equality and use our adoption that $t(G_{pp})=0$, we obtain the second equality since $G_{pqs}=G_{pq,ps}$.
\end{proof}
The results in \thmref{thm span vertex2and2} are analogues of \thmref{thm res cont1} and \thmref{thm res cont2}. 
The second equality in \thmref{thm span vertex2and2} can also be expressed in the following forms:
\begin{equation*}\label{}
\begin{split}
t(G) &= \frac{(a+b+c)^2-(a+b-c)^2-(a-b+c)^2-(-a+b+c)^2}{8t(G_{pqs})}\\
&= \frac{2a b+2bc+2ca-a^2-b^2-c^2}{4t(G_{pqs})},
\end{split}
\end{equation*}
where $ a=t(G_{p s})$, $b=t(G_{p q})$ and $c= t(G_{q s})$.

When we deal with the number of spanning trees, edge contraction can be seen as a special case of vertex identification because 
the self loops do not contribute to the number of spanning trees.
\begin{corollary}\label{cor span vertex contr}
Let $e_i$ be an edge of $G$ with end points $\pp$ and $\qq$. For any $s, \, t \in \vv{G}$, we have
\begin{equation*}\label{}
\begin{split}
 t(G) t(\overline{G}_{i,st})=t(G_{st}) t(\overline{G}_{i})  -\frac{1}{4}\big[ t(G_{p_i s})-t(G_{q_i s}) -t(G_{p_i t})+t(G_{q_i t}) \big ]^2.
\end{split}
\end{equation*}
\end{corollary}
\begin{proof}
Since $G_{p_iq_i}=\overline{G}_{i}$ and $G_{p_iq_i,st}=\overline{G}_{i,st}$, the result follows from \thmref{thm span vertex2and2}. 
\end{proof}
\corref{cor span vertex contr} explains how contracting an edge effects the number of spanning trees. The following results is its analogue for edge deletions:
\begin{theorem}\label{thm span edge del}
Let $e_i$ be an edge of $G$ with end points $\pp$ and $\qq$. For any $s, \, t \in \vv{G}$, we have
\begin{equation*}\label{}
\begin{split}
t(G) t((G-e_i)_{st})=t(G_{st}) t(G-e_i)  + \frac{1}{4}\big[ t(G_{p_i s})-t(G_{q_i s}) -t(G_{p_i t})+t(G_{q_i t}) \big ]^2.
\end{split}
\end{equation*}
\end{theorem}
\begin{proof}
The result follows from \thmref{thm del1} and \thmref{thm res and spantree}.

Alternatively, since $(G-e_i)_{st})=G_{st}-e_i$ and $\overline{G}_{i,st}=\overline{(G_{st})}_{i}$, the result follows by applying \thmref{thm spantree cont-del} to both $G$ and $G_{st}$,  and then using the result in \corref{cor span vertex contr}. 
\end{proof}

Next, we consider the graphs with vertex connectivity $3$ and obtain a formula for the number of spanning trees.
\begin{theorem}\label{thm span union v3}
Let $G=G_1 \cup G_2$ with $G_1 \cap G_2 =\{ p, \, q, \,s \}$ for any vertices $p$, $q$ and $s$ of $G$. Then
\begin{equation*}
\begin{split}
t(G)=&t(G_1) \cdot t(G_{2,pqs})+t(G_2) \cdot t(G_{1,pqs})\\
& \quad +\frac{1}{2}\Big[ a_1 (-a_2+b_2+c_2)+b_1(a_2-b_2+c_2)+c_1(a_2+b_2-c_2)   \Big], 
\end{split}
\end{equation*}
where $ a_i=t(G_{i,p s})$, $b_i=t(G_{i,p q})$ and $c_i= t(G_{i,q s})$ for any $i \in \{ 1, \, 2 \}$.

In particular, if $G_1=G_2$ and the same three vertices $p$, $q$ and $s$ are used in both copy of $G_1$, then we have
\begin{equation*}
\begin{split}
t(G)=4t(G_1) \cdot t(G_{1,pqs}).
\end{split}
\end{equation*}
\end{theorem}
\begin{proof}
We have $t(G_{pqs})=t(G_{1,pqs}) \cdot t(G_{2,pqs})$ by \thmref{thm span union1} as the identified vertex $pqs$ is a cut vertex in $G_{pqs}$. 
On the other hand, by \thmref{thm span union2} we have
\begin{equation*}
\begin{split}
t(G_{ps})&=t(G_{1,ps}) t(G_{2,pqs})+t(G_{2,ps})t(G_{1,pqs}),\\ 
t(G_{pq})&=t(G_{1,pq}) t(G_{2,pqs})+t(G_{2,pq})t(G_{1,pqs}), \quad \text{and}\\ 
t(G_{qs})&=t(G_{1,qs}) t(G_{2,pqs})+t(G_{2,qs})t(G_{1,pqs}).
\end{split}
\end{equation*}
Then the first equality in the theorem follows from the second equality in \thmref{thm span vertex2and2}.

The second equality in the theorem follows from the first one.
\end{proof}

\textbf{Example:} Let $G$ be as in the first graph in \figref{fig pqs}. Then $G_1$ and $G_{1,pqs}$ are as illustrated in the same figure. Note that $t(G_1)=27$, $t(G_{1,pqs})=45$, 
Thus, we have $t(G)=4t(G_1) \cdot t(G_{1,pqs})=4860$ by the second equality in \thmref{thm span union v3}.

\begin{figure}
\centering
\includegraphics[scale=0.55]{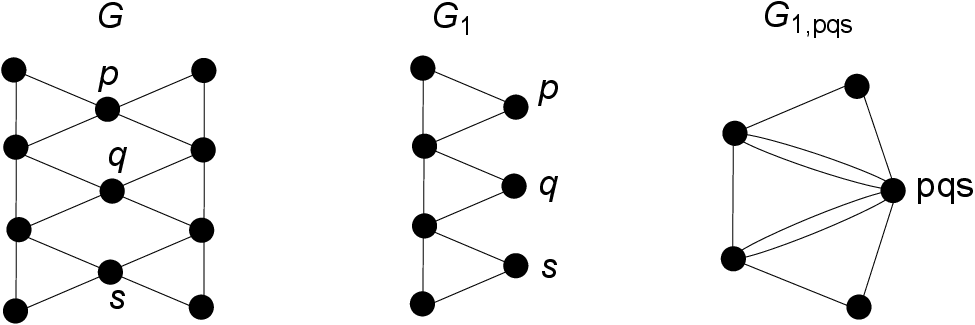} \caption{$G$ and the related graphs.}\label{fig pqs}
\end{figure}

\begin{theorem}\label{thm span vertex del1}
Let $p$, $q$, $s$ and $u$ be some vertices of a graph $G$, and let $u$ be adjacent to the vertices $p$, $q$ and $s$ via by $a$, $b$ and $c$ number of edges, respectively. If $u$ is not a cut vertex, we have
\begin{equation*}
\begin{split}
t(G)=(a+b+c) t(H)+ a c \, t(H_{ps})+ a b \, t(H_{pq})+ b c \, t(H_{qs}) + a b c \, t(H_{pqs}), 
\end{split}
\end{equation*}
where $H$ is the graph obtained from $G$ by deleting the vertex $u$.
\end{theorem}
\begin{proof}
Note that the deletion of a vertex requires also the deletions of all edges that are adjacent to that vertex. Since $u$ is not a cut vertex, $H$ is connected.
We take $G_1=H$ and $G_2$ as the graph with only the vertices $p$, $q$, $s$ and $u$ such that there are $a$ edges between $p$ and $u$, $b$ edges between $q$ and $u$, and $c$ edges connecting $s$ and $u$. Then $G=G_1 \cup G_2$ with 
$G_1 \cap G_2 =\{ p, \, q, \,s \}$.
Since $t(G_2)=a b c$, $t(G_{2,ps})=b (a+ c)$, $t(G_{2,pq})=c(a+ b)$, $t(G_{2,qs})=a(b+c)$, and $t(G_{2,pqs})=a+ b+c$, the result follows from \thmref{thm span union v3}.
\end{proof}

\textbf{Example:} Let $G=G_1 \cup G_2$ be as in the first graph in \figref{fig vdel1}, and let $G_1=G-u$, $G_{1,pqs}$ and $G_2$ be as illustrated in the same figure. Note that $t(G_1)=4$, $t(G_{1,pqs})=2$, $t(G_{1,ps})=4$, $t(G_{1,pq})=3$ and $t(G_{1,qs})=3$.
Thus, $t(G)=4(a+b+c)+ 4a c+ 3a b+ 3b c +2 a b c$ by \thmref{thm span vertex del1}, and we have $a=2$, $b=3$ and $c=1$ in this case. Hence, $t(G)=71$.

\begin{figure}
\centering
\includegraphics[scale=0.55]{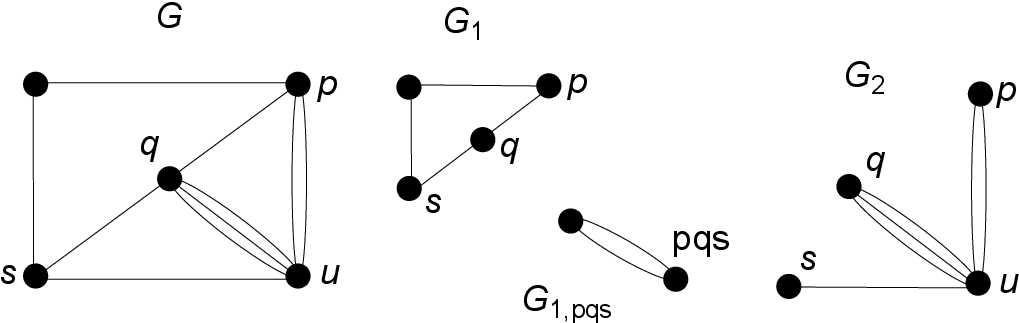} \caption{The graphs $G=G_1 \cup G_2$, $G_1=G-u$, $G_{1,pqs}$ and $G_2$.}\label{fig vdel1}
\end{figure}

We notice that \thmref{thm span vertex del0} and \thmref{thm span vertex del1} show some pattern, and this pattern is indeed also valid for a vertex $u$ that is not a cut vertex and adjacent to $4$ vertices. We use the notation $G-u$ to denote the graph obtained from $G$ by deleting the vertex $u$ and so by deleting the edges adjacent to $u$. When $u$ is not a cut vertex in $G$, $G-u$ is connected.
\begin{theorem}\label{thm span vertex del2}
Let $p_1$, $p_2$, $p_3$, $p_4$ and $u$ be some vertices of a graph $G$, and let $u$ be adjacent to the vertex $p_i$ via by $a_i$ number of edges for each $i \in \{1, \, 2, \, 3, \, 4 \}$. If $u$ is not a cut vertex, then for $G$ and $H=G-u$ we have
\begin{equation*}
\begin{split}
t(G)&=(a_1+a_2+a_3+a_4) t(H)+ a_1 a_2 \, t(H_{p_1p_2})+ a_1 a_3 \, t(H_{p_1p_3})+ a_1 a_4 \, t(H_{p_1p_4}) \\ 
& \quad + a_2 a_3 \, t(H_{p_2p_3}) +a_2 a_4 \, t(H_{p_2p_4})+a_3 a_4 \, t(H_{p_3p_4})+ a_1 a_2 a_3 \, t(H_{p_1p_2p_3}) \\
& \quad +a_1 a_2 a_4 \, t(H_{p_1p_2p_4})+a_1 a_3 a_4 \, t(H_{p_1p_3p_4})
 +a_2 a_3 a_4 \, t(H_{p_2p_3p_4})+a_1 a_2 a_3 a_4 \, t(H_{p_1p_2p_3p_4}).
\end{split}
\end{equation*}
\end{theorem}
\begin{proof}
By the second equality in \thmref{thm span vertex2and2} we have
\begin{equation}\label{eqn span vertex del2a}
\begin{split}
t(G)&=\frac{1}{G_{p_1 p_2 p_3}} \Big[ t(G_{p_1 p_2}) t(G_{p_1 p_3})-\frac{1}{4} \big( t(G_{p_1 p_2})+ t(G_{p_1 p_3}) -t(G_{p_2 p_3} ) \big)^2  \Big].
\end{split}
\end{equation}
Let $A$ and $B$ denote the right hand sides of the equalities in \eqref{eqn span vertex del2a} and \thmref{thm span vertex del2}, respectively. We prove the theorem by showing $A-B=0$. 

We set $h_{ij}=t(H_{p_ip_j})$ and $h_{ijk}=t(H_{p_ip_jp_k})$, where $i$, $j$ and $k$ are elements from $\{1, \,  2, \, 3, \, 4  \}$. We also set $h=t(H)$, $h_{12, \, 34}=t(H_{p_1p_2, \, p_3p_4})$, $h_{13, \, 24}=t(H_{p_1p_3, \, p_2p_4})$, $h_{14, \, 23}=t(H_{p_1p_4, \, p_2p_3})$ and $h_{1234}=t(H_{p_1p_2p_3p_4})$.

To show that $A-B=0$, we express each term in $A$ and $B$ in terms of $h$ and $h_{ij}$'s.

Since $G_{p_1p_2p_3}-u=H_{p_1p_2p_3}$ and $(G_{p_1p_2p_3}-u)_{p_1p_2p_3p_4}=H_{p_1p_2p_3p_4}$,  we use \thmref{thm span vertex del0} to derive
\begin{equation}\label{eqn span vertex del2b1}
\begin{split}
t(G_{p_1p_2p_3})=(a_1+a_2+a_3+a_4) h_{123}+(a_1+a_2+a_3)a_4 \, h_{1234}.
\end{split}
\end{equation}
Since $G_{p_1p_2}-u=H_{p_1p_2}$, $(G_{p_1p_2}-u)_{p_1p_2p_3}=H_{p_1p_2p_3}$, $(G_{p_1p_2}-u)_{p_1p_2p_4}=H_{p_1p_2p_4}$,
$(G_{p_1p_2}-u)_{p_3p_4}=H_{p_1p_2, \, p_3p_4}$ and  $(G_{p_1p_2}-u)_{p_1p_2p_3p_4}=H_{p_1p_2p_3p_4}$  we use \thmref{thm span vertex del1} to derive
\begin{equation}\label{eqn span vertex del2b2}
\begin{split}
t(G_{p_1p_2})&=(a_1+a_2+a_3+a_4) h_{12}+a_3(a_1+a_2) h_{123}+a_4(a_1+a_2) h_{124}+a_1a_4 h_{12, \, 34}\\
& \quad +a_3 a_4(a_1+a_2) h_{1234}.
\end{split}
\end{equation}
Similarly, we have
\begin{equation}\label{eqn span vertex del2b3}
\begin{split}
t(G_{p_1p_3})&=(a_1+a_2+a_3+a_4) h_{13}+a_2(a_1+a_3) h_{123}+a_4(a_1+a_3) h_{134}+a_2a_4 \, h_{13,\, 24}\\
& \quad +a_2 a_4(a_1+a_3) h_{1234},\\
t(G_{p_2p_3})&=(a_1+a_2+a_3+a_4) h_{23}+a_1(a_2+a_3) h_{123}+a_4(a_2+a_3) h_{234}+a_1a_4 \, h_{14,\, 23}\\
& \quad +a_1 a_4 (a_2+a_3) h_{1234}.
\end{split}
\end{equation}
Applying the second equality in \thmref{thm span vertex2and2} to $H$ and $H_{p_1p_2}$ gives
\begin{equation}\label{eqn span vertex del2b4}
\begin{split}
h_{123}&=\frac{1}{h}\Big[h_{12} h_{13}-\frac{1}{4} \big( h_{12}+ h_{13}-h_{23} \big)^2 \Big],\\
h_{124}&=\frac{1}{h}\Big[h_{12} h_{14}-\frac{1}{4} \big( h_{12}+ h_{14}-h_{24} \big)^2 \Big],\\
h_{134}&=\frac{1}{h}\Big[h_{13} h_{14}-\frac{1}{4} \big( h_{13}+ h_{14}-h_{34} \big)^2 \Big],\\
h_{234}&=\frac{1}{h}\Big[h_{23} h_{24}-\frac{1}{4} \big( h_{23}+ h_{24}-h_{34} \big)^2 \Big],\\
h_{1234}&=\frac{1}{h_{12}}\Big[h_{123} h_{124}-\frac{1}{4} \big( h_{123}+ h_{124}-h_{12, \, 34} \big)^2 \Big].
\end{split}
\end{equation}
On the other hand, applying the first equality in \thmref{thm span vertex2and2} to $H$ gives
\begin{equation}\label{eqn span vertex del2b5}
\begin{split}
h_{12,\, 34}&=\frac{1}{h}\Big[h_{12} h_{34}-\frac{1}{4} \big( h_{13}-h_{23}-h_{14}+h_{24} \big)^2 \Big],\\
h_{13,\, 24}&=\frac{1}{h}\Big[h_{13} h_{24}-\frac{1}{4} \big( h_{12}-h_{23}-h_{14}+h_{34} \big)^2 \Big],\\
h_{14,\, 23}&=\frac{1}{h}\Big[h_{14} h_{23}-\frac{1}{4} \big( h_{12}-h_{13}-h_{24}+h_{34} \big)^2 \Big].\\
\end{split}
\end{equation}
Finally, substituting the values in \eqref{eqn span vertex del2b1}, \eqref{eqn span vertex del2b2}, \eqref{eqn span vertex del2b3}, \eqref{eqn span vertex del2b4} and \eqref{eqn span vertex del2b5} in $A-B$ and doing the algebra by \cite{MMA} gives $0$.
This completes the proof.
\end{proof}
The pattern observed in \thmref{thm span vertex del0}, \thmref{thm span vertex del1} and \thmref{thm span vertex del2} is in fact 
valid for any non-cut vertex in a graph. To prove this for a general case, we need the following result:
\begin{theorem}\label{thm span vertex del general0}
Let $A= \{ p_1, \, \ldots, \, p_{n} \} \subset V(H)$ for a graph $H$, and let $\bar{H}$ be a graph obtained from $H$ by adding $a_i \geq 1$ edges connecting the vertices $p_n$ and $p_i$ for each $i \in \{1, \, \ldots, \, n-1 \}$ with $n \geq 2$.
Then, we have
\begin{equation*}
\begin{split}
t(\bar{H})&= t(H) + \sum_{\substack{p_n \in S \subset A \\ |S| \geq 2}} \Big( \prod_{i \in I_{S-\{ p_n \}} } a_i \Big) t(H_{S}),
\end{split}
\end{equation*}
where $I_{S-\{ p_n \}}$ is the set of indexes of the vertices in $S-\{ p_n \}$, and  $H_S$ is the graph obtained from $H$ by identifying all vertices in $S$.
\end{theorem}
\begin{proof}
The proof is given by induction on $k=|A|$.

Let $k=2$. We apply  \thmref{thm spantree cont-del} successively to obtain 
$$t(\bar{H})= t(H)+a_1 t(H_{p_1p_2}),$$ 
as $H_{p_1p_2}=\bar{H}_{p_1p_2}$. This is what we want in this case.

Let $k=3$. We can apply  \thmref{thm spantree cont-del} to each added edge connecting $p_1$ and $p_3$ successively so that we obtain $t(\bar{H})= t(J)+a_1 t(J_{p_1p_3})$, where $J$ is the graph obtained from $\bar{H}$ by deleting all $a_1$ edges that are added to $H$ to connect $p_1$ and $p_3$ as we obtain $\bar{H}$. We can apply the case $n=2$ to both $J$ and $J_{p_1p_3}$ so that we obtain $t(J)=t(H)+a_2 t(J_{p_2 p_3})=t(H)+a_2 t(H_{p_2 p_3})$, and 
$t(J_{p_1p_3})=t(H_{p_1p_3})+a_2 t(H_{p_1 p_2 p_3})$ since identifying the vertices $p_1p_3$ and $p_2$ in $J_{p_1p_3}$ gives the graph $H_{p_1 p_2 p_3}$. Therefore, combining these equations give 
$$t(\bar{H})=t(H)+a_2 t(H_{p_2 p_3})+a_1t(H_{p_1p_3})+a_1 a_2 t(H_{p_1 p_2 p_3}).$$
This proves the given identity for $k=3$.

Let $k$ be an integer with $k\geq 2$. Assume that the given identity holds for any graphs $F$ and $\bar{F}$ 
so that $\bar{F}$ contains the extra set of edges connecting the vertices in $\{ p_1, \, \ldots, \, p_{k} \} \subset V(F)$ 
as described in the theorem. 

Suppose the graphs $H$ and $\bar{H}$ are as in the theorem with $|A|=k+1$, in which case we have $A=\{ p_1, \, \ldots, \, p_{k}, \, p_{k+1} \} \subset V(H)$. Suppose $K$ is the graph obtained from $\bar{H}$ by deleting all $a_1$ edges that are added to $H$ to connect $p_1$ and $p_{k+1}$ as we obtain $\bar{H}$. We can apply  \thmref{thm spantree cont-del} to each added edge connecting $p_1$ and $p_{k+1}$ successively so that we obtain 
$$t(\bar{H})= t(K)+a_1 t(K_{p_1p_{k+1}}).$$
On the other hand, the following identities for $K$ and $K_{p_1p_{k+1}}$ follows from the induction assumption:
$$t(K)= t(H) + \sum_{\substack{p_{k+1} \in S \subset A-\{ p_1 \} \\ |S| \geq 2}} \Big( \prod_{i \in I_{S-\{ p_{k+1} \}} } a_i \Big) t(H_{S}),
$$
$$
t(K_{p_1p_{k+1}})= t(H_{p_1p_{k+1}}) + \sum_{\substack{p_1p_{k+1} \in S \subset A' \\ |S| \geq 2}} \Big( \prod_{i \in I_{S-\{ p_1 p_{k+1} \}} } a_i \Big) t((H_{p_1p_{k+1}})_{S}),
$$
where $A'=\{ p_2, \, p_3, \, \ldots, \, p_{k}, \, p_1p_{k+1} \} \subset V(H_{p_1p_{k+1}})$.
Combining these identities gives
\begin{equation*}
\begin{split}
t(\bar{H})&= t(H) + \sum_{\substack{p_{k+1} \in S \subset A-\{ p_1 \} \\ |S| \geq 2}} \Big( \prod_{i \in I_{S-\{ p_{k+1} \}} } a_i \Big) t(H_{S}) + a_1 t(H_{p_1p_{k+1}}) \\
&\quad + a_1 \sum_{\substack{p_1p_{k+1} \in S \subset A' \\ |S| \geq 2}} \Big( \prod_{i \in I_{S-\{ p_1 p_{k+1} \}} } a_i \Big) t((H_{p_1p_{k+1}})_{S})\\
&= t(H) + \sum_{\substack{p_{k+1} \in S \subset A \\ |S| \geq 2}} \Big( \prod_{i \in I_{S-\{ p_{k+1} \}} } a_i \Big) t(H_{S}).
\end{split}
\end{equation*}
This is what we wanted for the case $k+1$. Hence the result follows by induction principle.
\end{proof}

The following result shows how the number of spanning trees changes if a vertex is deleted. This is one of the main result of this paper.
\begin{theorem}\label{thm span vertex del general}
Let $u \in V(G)$, $N_G(u)= \{ p_1, \, \ldots, \, p_n \} \subset V(G)$ for a graph $G$, and let $u$ be adjacent to the vertex $p_i$ via by $a_i \geq 1$ number of edges for each $i \in \{1, \, \ldots, \, n \}$ with $n \geq 2$. If $u$ is not a cut vertex, then for $G$ and $H=G-u$ we have
\begin{equation*}
\begin{split}
t(G)&= \Big( \sum_{i=1}^n a_i \Big) t(H) + \sum_{\substack{S \subset N_G(u) \\ |S| \geq 2}} \Big( \prod_{i \in I_S } a_i \Big) t(H_{S}),
\end{split}
\end{equation*}
where $I_S$ is the set of indexes of the vertices in $S$, and  $H_S$ is the graph obtained from $H$ by identifying all vertices in $S$. 
\end{theorem}
\begin{proof}
The proof is by induction on $k=|N_G(u)|$.
When $k=2, \, 3, \, 4$, the results are given in \thmref{thm span vertex del0}, \thmref{thm span vertex del1} and \thmref{thm span vertex del2}, respectively.

Assume that the given identity holds for any graph $F$ with a vertex $s \in V(F)$ having the set of adjacent vertices $N_F(s)=\{ p_1, \, \ldots, \, p_k \}$, where $k$ is an integer with $k \geq 4$. 
Then we prove the given identity for a graph $G$ with a vertex $u \in V(G)$ having the set of adjacent vertices $N_G(u)=\{ p_1, \, \ldots, \, p_k, \, p_{k+1} \}$.

We have $a_{k+1}$ edges $e_{k+1,j}$ connecting $u$ and $p_{k+1}$ in $G$, where $j \in \{ 1, \, 2, \, \ldots, a_{k+1} \}$. We apply the deletion-contraction identity given in \thmref{thm spantree cont-del} for each of these edges. This gives  
\begin{equation}\label{eqn part del-cont} 
\begin{split}
t(G)&=t(G-e_{k+1,1})+t(\overline{G}_{k+1,1}), \quad \text{by applying \thmref{thm spantree cont-del} to $e_{k+1,1}$} \\
&= t(G-e_{k+1,1})+t(G_{p_{k+1} u}), \quad \text{since $\overline{G}_{k+1,1}=G_{p_{k+1} u}$}\\
&= t(G-e_{k+1,1}-e_{k+1,2})+2t(G_{p_{k+1} u}), \quad \text{repeating the same process for  $e_{k+1,2}$}\\
&\ldots\\
&=t(G')+a_{k+1}t(G_{p_{k+1} u}),
\end{split}
\end{equation}
where $G'=G-e_{k+1,1}-e_{k+1,2}- \, \cdots \, -e_{k+1,a_{k+1}}$ is the graph obtained from $G$ by deleting all edges connecting $u$ and $p_{k+1}$. Thus by the induction assumption for $G'$ and $u$ we have
\begin{equation}\label{eqn part del}
\begin{split}
t(G')&= \Big( \sum_{i=1}^k a_i \Big) t(H) + \sum_{\substack{S \subset N_{G'}(u) \\ |S| \geq 2}} \Big( \prod_{i \in I_S } a_i \Big) t(H_{S}), \quad \text{since $G'-u=G-u=H$}\\
&= \Big( \sum_{i=1}^k a_i \Big) t(H) + \sum_{\substack{S \subset N_{G}(u)-\{p_{k+1} \} \\ |S| \geq 2}} \Big( \prod_{i \in I_S } a_i \Big) t(H_{S}).
\end{split}
\end{equation}
On the other hand, $G_{p_{k+1} u}$ is nothing but $\bar{H}$ which is obtained from $H$ by adding edges between the specified set of vertices among $N_{G}(u)=\{ p_1, \, \ldots, \, p_k, \, p_{k+1} \}$ as in \thmref{thm span vertex del general0}. Thus, applying this theorem gives
\begin{equation}\label{eqn part cont}
\begin{split}
t(G_{p_{k+1} u})= t(H) + \sum_{\substack{p_{k+1} \in S \subset N_{G}(u) \\ |S| \geq 2}} \Big( \prod_{i \in I_{S-\{ p_{k+1} \}} } a_i \Big) t(H_{S}),
\end{split}
\end{equation}
Then using \eqref{eqn part del} and \eqref{eqn part cont} in \eqref{eqn part del-cont} gives
\begin{equation*}
\begin{split}
t(G)&=\Big( \sum_{i=1}^{k+1} a_i \Big) t(H)+\sum_{\substack{S \subset N_{G}(u)-\{p_{k+1} \} \\ |S| \geq 2}} \Big( \prod_{i \in I_S } a_i \Big) t(H_{S})\\
&\quad +a_{k+1} \sum_{\substack{p_{k+1} \in S \subset N_{G}(u) \\ |S| \geq 2}} \Big( \prod_{i \in I_{S-\{ p_{k+1} \}} } a_i \Big) t(H_{S}) \\
&=\Big( \sum_{i=1}^{k+1} a_i \Big) t(H)+\sum_{\substack{S \subset N_G(u) \\ |S| \geq 2}} \Big( \prod_{i \in I_S } a_i \Big) t(H_{S}).
\end{split}
\end{equation*}
This is what we wanted for $k+1$. Thus, the proof of the theorem follows by induction principle.
\end{proof}

It is not difficult to see that \thmref{thm span vertex del general} will have many interesting applications.
We apply this theorem to determine the number of of spanning trees of some families of graphs. First, we recall and make some definitions as a preparation.

Morgan-Voyce Polynomials $B_n(x)$ are defined as follows (\cite{MVo}):
\begin{equation*}\label{eqn MV}
\begin{split}
&B_0(x)=1, \quad B_1(x)=2+x, \\
&B_n(x)=(x+2)B_{n-1}(x)-B_{n-2}(x), \quad \text{for $n \geq 2$}.
\end{split}
\end{equation*}
We have (see \cite{Sw1})
\begin{equation}\label{eqn MV20}
\begin{split}
B_n(x)= \sum_{k=0}^n \binom{n+k+1}{n-k}x^k. 
\end{split}
\end{equation}
It is known that $x B_n(x^2)=F_{2n+2}(x)$ (\cite{Sw2}), 
where $F_n(x)$ is the $n$-th Fibonacci polynomial.

We define the modified fan graph $\overline{Fan}_n$ as follows: Given the path graph $P_n$ with vertex set $\{ p_1, \cdots, p_n \}$, we join a vertex $p$ to $P_n$ by adding $a \geq 1$ multiple edges between $p$ and any vertex of $P_n$. For example,  $\overline{Fan}_2$ with $a=1$ is nothing but the cycle graph $C_3$, and $\overline{Fan}_n$ with $a=1$ is known as the fan graph $Fan_n$ with $n+1$ vertices. The number of spanning trees of $Fan_n$ has studied in literature (\cite{Bog},  \cite{HB}). Here, we give a new proof for it by using \thmref{thm span vertex del general}. Moreover, we generalize this result for $\overline{Fan}_n$:

\begin{theorem}\label{thm fan}
For any integers $a \geq 1$ and $n \geq 2$, we have
\begin{equation*}\label{eqn MV2}
\begin{split}
t(\overline{Fan}_n)=a B_{n-1}(a).
\end{split}
\end{equation*}
In particular, when $a=1$, $t(Fan_n)= F_{2n}$, where $F_n$ is the $n$-th Fibonacci number.
\end{theorem}
\begin{proof}
Let $H$ be the graph obtained from $\overline{Fan}_n$ by deleting the vertex $p$. Then we have $H=P_n$.
Let $H_{j_1j_2 \cdots j_k}$ be the graph obtained from $H$ by identifying the vertices $p_{j_1}$, $p_{j_2}$, ..., $p_{j_k}$ where the indexes satisfy the inequality $j_1 < j_2< \cdots < j_k $ for some integer $2 \leq k \leq n$. Then we have 
$$
t(H_{j_1j_2 \cdots j_k})=(j_2-j_1)(j_3-j_2) \cdots (j_k-j_{k-1}).
$$
Thus, for $S=\{1, 2, \cdots, n \}$,
\begin{equation}\label{eqn MV2a}
\begin{split}
\sum_{\{ j_1,j_2,\cdots, j_k \} \subset S}t(H_{j_1j_2 \cdots j_k})=\sum_{1 \leq j_1 <j_2< \cdots < j_k \leq n }(j_2-j_1)(j_3-j_2) \cdots (j_k-j_{k-1})=\binom{n-1+k}{2k-1}.
\end{split}
\end{equation}
Alternatively, we have
$$\sum_{\{ j_1,j_2,\cdots, j_k \} \subset S}t(H_{j_1j_2 \cdots j_k})=T(n,k),$$
where $T(n,k)$ are the triangular numbers (see \cite{O1}) given by
\begin{equation*}\label{eqn MVb}
\begin{split}
&T(0,0)=1, \quad \text{and} \quad T(n,k)=0, \quad \text{if  $k > n$ or $k<0$}\\
&T(n,k)=T(n-1,k-1)+2T(n-1,k)-T(n-2,k), \quad \text{if  $n < k$}
\end{split}
\end{equation*}
On the other hand, by \thmref{thm span vertex del general}, we have
\begin{equation*}\label{eqn MV2c}
\begin{split}
t(\overline{Fan}_n)&=n \cdot a \cdot t(H)+\sum_{k=2}^n a^k \sum_{\{ j_1,j_2,\cdots, j_k \} \subset S}
t(H_{j_1j_2 \cdots j_k})\\
&=na+\sum_{k=2}^n a^k \binom{n-1+k}{2k-1}, \quad \text{by \eqref{eqn MV2a}}\\
&=\sum_{k=1}^n a^k \binom{n-1+k}{n-k}\\
&=\sum_{k=0}^{n-1} a^{k+1} \binom{n+k}{n-1-k}\\
&=a B_{n-1}(a), \quad \text{by \eqref{eqn MV20}}.
\end{split}
\end{equation*}
This completes the proof.
\end{proof}
Note that the proof of \thmref{thm fan} gives more insight to the formula of $t(\overline{Fan}_n)$. Namely, it also expresses  $t(\overline{Fan}_n)$ as sum of the triangular numbers, which are the coefficients of the Morgan-Voyce Polynomial $B_n(x)$ in this case. When $a=1$, we have the following triangular numbers $T(n,k)$, $B_n(x)$ and $t(\overline{Fan}_n)$ for small values of $n$:
\begin{equation*}\label{eqn MVd}
\begin{split}
&1, \quad B_0(x)=1, \quad t(\overline{Fan}_1)=1= F_{2}\\
&2, \, \, 1, \quad B_1(x)=2+x, \quad t(\overline{Fan}_2)=3= F_{4}\\
&3, \, \, 4, \, \, 1, \quad B_2(x)=3+4x+x^2, \quad t(\overline{Fan}_3)=8= F_{6}\\
&4, \, \, 10, \, \, 6, \, \, 1, \quad B_3(x)=4+10x+6x^2+x^3, \quad t(\overline{Fan}_4)=21= F_{8}\\ 
&5, \, \, 20, \, \, 21, \, \, 8, \, \, 1, \quad B_4(x)=5+20x+21x^2+8x^3+x^4, \quad t(\overline{Fan}_5)=55=F_{10}.
\end{split}
\end{equation*}

We give another application of \thmref{thm span vertex del general}. Again, we first recall and make some definitions.

We define the polynomials $W_n(x)$ as follows:
\begin{equation*}\label{eqn C}
\begin{split}
&W_0(x)=1, \quad W_1(x)=x+4, \\
&W_{n}(x)=(x+2)W_{n-1}(x)-W_{n-2}(x)+2, \quad \text{for $n \geq 2$}.
\end{split}
\end{equation*}
One can use the strong induction on $n$ and the recursive definitions of $W_n(x)$ to prove
\begin{equation}\label{eqn C0}
\begin{split}
W_n(x)= \sum_{k=0}^n \frac{2n+2}{n+2+k}\binom{n+2+k}{n-k} x^k. 
\end{split}
\end{equation}
We note that $x^2 W_n(x^2)= L_{2n+2}(x)-2$ (see \cite[pg. 16]{K} for the summation formula for $L_{n}(x)$), 
where $L_n(x)$ is the $n$-th Lucas polynomial. Another closely related sequence of polynomials is $C_n(x)$, one of the companion Morgan-Voyce polynomials (see \cite[pg. 165 ]{Ho}). Namely, we have $x W_{n-1}(x)+2=C_n(x)$. 
%

We define the modified wheel graph $\overline{W}_n$ as follows: Given the cycle graph $C_n$ with vertex set $\{ p_1, \cdots, p_n \}$, we join a vertex $p$ to $C_n$ by adding $a \geq 1$ multiple edges between $p$ and any vertex of $C_n$. For example,  $\overline{W}_2$ with $a=2$ is the graph obtained from the cycle graph $C_3$ by replacing each of its edges by double edges, and $\overline{W}_n$ with $a=1$ is known as the wheel graph $W_n$ with $n+1$ vertices. The number of spanning trees of $W_n$ has studied in literature (\cite{My}, \cite{HB} and \cite{S}). Here, we give a new proof for it by using \thmref{thm span vertex del general}. Moreover, we generalize this result for $\overline{W}_n$:

\begin{theorem}\label{thm wheel}
For any integers $a \geq 1$ and $n \geq 2$, we have
\begin{equation*}\label{eqn CC}
\begin{split}
t(\overline{W}_n)=a W_{n-1}(a).
\end{split}
\end{equation*}
In particular, when $a=1$, $t(W_n)= L_{2n}-2$, where $L_n$ is the $n$-th Lucas number.
\end{theorem}
\begin{proof}
Let $H$ be the graph obtained from $\overline{W}_n$ by deleting the vertex $p$. Then we have $H=C_n$.
Let $H_{j_1j_2 \cdots j_k}$ be the graph obtained from $H$ by identifying the vertices $p_{j_1}$, $p_{j_2}$, ..., $p_{j_k}$ where the indexes satisfy the inequality $j_1 < j_2< \cdots < j_k $ for some integer $2 \leq k \leq n$. Then we have 
$$
t(H_{j_1j_2 \cdots j_k})=(j_2-j_1)(j_3-j_2) \cdots (j_k-j_{k-1}) (j_1+n-j_{k}).
$$
Thus, for $S=\{1, 2, \cdots, n \}$, we have the following equality for each $2 \leq k \leq n$:
\begin{equation}\label{eqn Ca}
\begin{split}
\sum_{\{ j_1,j_2,\cdots, j_k \} \subset S}t(H_{j_1j_2 \cdots j_k})&=\sum_{1 \leq j_1 <j_2< \cdots < j_k \leq n }(j_2-j_1)(j_3-j_2) \cdots (j_k-j_{k-1}) (j_1+n-j_{k}) \\
&= \frac{2n}{n+k}\binom{n+k}{n-k}.
\end{split}
\end{equation}
Alternatively, for each $2 \leq k \leq n$ we have 
$$\sum_{\{ j_1,j_2,\cdots, j_k \} \subset S}t(H_{j_1j_2 \cdots j_k})=T(n,n-k),$$
where $T(n,k)$ are the triangular numbers (see \cite{O2}) such that the row $n$ has $n+1$ terms and $n$-th row is given by 
\begin{equation*}\label{eqn Cb}
\begin{split}
&T(n,0)=1, \quad \text{and} \quad T(n,n)=2, \quad \text{for $n >0$}\\
&T(n,k)=T(n-1,k-1)+\sum_{i=0}^k T(n-1-i,n-i), \quad \text{if $0 < k < n$}
\end{split}
\end{equation*}
On the other hand, by \thmref{thm span vertex del general}, we have
\begin{equation*}\label{eqn Cc}
\begin{split}
t(\overline{W}_n)&=n \cdot a \cdot t(H)+\sum_{k=2}^n a^k \sum_{\{ j_1,j_2,\cdots, j_k \} \subset S}
t(H_{j_1j_2 \cdots j_k})\\
&=n^2a+\sum_{k=2}^{n} a^{k} \frac{2n}{n+k} \binom{n+k}{n-k}, \quad \text{by \eqref{eqn Ca}}\\
&=\sum_{k=1}^n a^k \frac{2n}{n+k} \binom{n+k}{n-k}\\
&=\sum_{k=0}^{n-1} a^{k+1} \frac{2n}{n+1+k} \binom{n+1+k}{n-1-k}\\
&=a W_{n-1}(a), \quad \text{by \eqref{eqn C0}}.
\end{split}
\end{equation*}
This completes the proof.
\end{proof}

Again, the proof of \thmref{thm wheel} gives more insight to the formula of $t(\overline{W}_n)$. Namely, it also expresses  $t(\overline{W}_n)$ as sum of the triangular numbers, which are the coefficients of the polynomials $W_n(x)$ in this case. When $a=1$, we have the following triangular numbers $T(n,k)$, $W_n(x)$ and $t(\overline{W}_n)$ for small values of $n$:
\begin{equation*}\label{eqn Cd}
\begin{split}
&1, \quad W_0(x)=1, \quad t(\overline{W}_1)=1= L_{2}-2\\
&4, \, \, 1, \quad W_1(x)=4+x, \quad t(\overline{W}_2)=5= L_{4}-2\\
&9, \, \, 6, \, \, 1, \quad W_2(x)=9+6x+x^2, \quad t(\overline{W}_3)=16= L_{6}-2\\
&16, \, \, 20, \, \, 8, \, \, 1, \quad W_3(x)=16+20x+8x^2+x^3, \quad t(\overline{W}_4)=45= L_{8}-2\\ 
&25, \, \, 50, \, \, 35, \, \, 10, \, \, 1, \quad W_4(x)=25+50x+35x^2+10x^3+x^4, \quad t(\overline{W}_5)=121=L_{10}-2.
\end{split}
\end{equation*}

Some more applications of  \thmref{thm span vertex del general} will be given in a subsequent paper.

\textbf{Declaration of competing interest:} The author declares that he has no known competing financial interest or personal relationship that could have appeared to influence the work reported in this paper.

\end{document}